\documentclass{amsart}
\usepackage{amsmath} 
\usepackage{amsthm}
\usepackage{amssymb}
\usepackage{mathrsfs}
\usepackage{hyperref}
\usepackage{color}
\usepackage{todonotes}

\newtheorem{theorem}{Theorem} [section]

\newtheorem{lemma}[theorem]{Lemma}
\newtheorem{proposition}[theorem]{Proposition}
\newtheorem{corollary}[theorem]{Corollary}
\newtheorem{definition}[theorem]{Definition}

\numberwithin{equation}{section}

\begin{document}
	\title{Liv\v{s}ic theorems for Banach cocycles: existence and regularity}
	
\author{Rui Zou}
\address{School of Mathematics and Statistics, Nanjing University of Information Science and Technology, Nanjing 210044,  P.R. China}
\email{zourui@nuist.edu.cn }

\author{Yongluo Cao*}
\address{Departament of Mathematics, Shanghai Key Laboratory of PMMP, East China  Normal University,	Shanghai 200062, P.R. China}
\address{Departament of Mathematics, Center for Dynamical systems and Differential Equations, Soochow University,Suzhou 215006, Jiangsu, P.R. China}
\email{ylcao@suda.edu.cn}

\thanks{* Yongluo Cao is corresponding author.  This work was partially supported by NSFC (11771317,  11790274, 11901305),  Science and Technology Commission
	of Shanghai Municipality (18dz22710000).}

\date{\today}

\begin{abstract}
 We prove a nonuniformly hyperbolic version Liv\v{s}ic theorem, with cocycles taking values in  the group of invertible bounded linear operators on a Banach space. The result holds without the ergodicity assumption of the hyperbolic measure.  Moreover, We also prove a $\mu$-continuous solution of the cohomological equation is actually H\"older continuous for the uniform hyperbolic system.
\end{abstract}


\maketitle

 \section{Introduction}   
   
   For a given dynamical system $f:M\to M$ and a  map $A:M\to G$, where G is a topological group, it's important to determine whether A is a \emph{coboundary}, that is, whether there exists a map $C:M\to G$ such that 
   $$A=(C\circ f)\cdot C^{-1}.$$
   Such a equation is usually called the \emph{cohomological equation} and $C$ is  a solution to the equation.
   
   These problems were first studied by Liv\v{s}ic \cite{livsic1971,livsic1972}. He proved that if $f$ is a hyperbolic system, $G=\mathbb{R}$ and A is H\"older continuous, then A is coboundary, if and only if 
   \[\sum\limits_{i=0}^{n-1}A(f^ip)=0, \quad\forall p=f^n(p), n\geq 1. \]
   
   Due to the interest and the importance of this result, many generalizations have been studied in different directions:
   \begin{enumerate}
   	\item More general groups: Does the Liv\v{s}ic theorem hold for more general groups?
   	\item More general dynamics: For other dynamical systems, e.g., nonuniformly hyperbolic systems, partially hyperbolic systems, etc., is there a Liv\v{s}ic-type theorem?
   	\item Regularity of solutions: If the cohomological equation has a measurable solution, does it coincide almost everywhere with a continuous one? Is  a continuous solution actually $C^r$?
   \end{enumerate}

   Around these questions, the cohomological equations has been extensively studied in recent decades. We introduce some of the highlights from different dynamics:

     $\bullet$ \emph{Expanding systems}. Conze, Guivarc'h \cite{conze1993croissance} and Savchenko \cite{savchenko1998cohomology} proved a \emph{non-positive Liv\v{s}ic theorem}, that is, if $A:M\to \mathbb{R}$ satisfies $\sum\limits_{i=0}^{n-1}A(f^ip)\leq0, \forall p=f^n(p), n\geq 1,$ then there exists a H\"older continuous function $C:M\to \mathbb{R}$ such that $A\leq C\circ f-C$.

     $\bullet$ \emph{Uniformly hyperbolic systems}. For $G=\mathbb{R}$, Bousch \cite{bousch2001condition} and  Lopes and Thieullen \cite{lopes2003} proved the non-positive Liv\v{s}ic theorem.   Liv\v{s}ic \cite{livsic1971}  also proved the Liv\v{s}ic theorem when  the group G admits a
     complete bi-invariant distance( e.g. Abelian or compact groups). For the group G not admitting  bi-invariant distances, one of the main difficulties is to "control distortions". To do so, initially, many authors\cite{livsic1972,pollicott2001,de2010livsic} assumed that the cocycle is sufficiently close to the constant identity cocycle.   A first breakthrough progress, without additional hypotheses, was made by Kalinin \cite{kalinin2011} in the case when $G=GL(d,\mathbb{R})$.  And then, Grabarnik and Guysinsky \cite{grabarnik} generalized  the result to Banach rings. Navas and Ponce \cite{Navas13} considered the group of germs of analytic diffeomorphisms. For groups of diffeomorphisms,  Kocsard and Potrie \cite{kocsard2016} and Avila, Kocsard and Liu \cite{AvilaLiu17} proved the corresponding  Liv\v{s}ic theorem. On the regularity of solutions of the cohomological equation, for the connected Lie group, Pollicott and Walkden \cite{pollicott2001} 
     proved under the "partial hyperbolicity" condition that every measurable solution coincides almost everywhere with a  H\"older continuous one. Bulter \cite{Butlerconformal} considered the case of the group $G=GL(d,\mathbb{R})$.
     
      $\bullet$ \emph{Flows}. For a transitive Anosov flow and $G=\mathbb{R}$,  the classical Liv\v{s}ic theorem was established by Liv\v{s}ic \cite{livsic1971}. Pollicott and Walkden  generalized the result to connected Lie groups in \cite{pollicott2001}. The non-positive Liv\v{s}ic theorem for $G=\mathbb{R}$ was proved by Pollicott and Sharp \cite{pollicott2004flow} and Lopes and Thieullen \cite{lopes2005}.
      
      $\bullet$ \emph{Partially hyperbolic systems}. In a partially hyperbolic system, since the periodic orbit may not exist, Katok and Kononenko \cite{Katok96} used the "periodic cycle function" to replace the periodic points, and they gave a sufficient and necessary condition of the coboundary of $A:M\to \mathbb{R}$ when the system $f$ is locally accessible.  And then Wilkinson \cite{Wilkinson13} generalized the result to $f$ is accessible, she also considered the regularity of solutions of the cohomological equations. The result for Banach cocycles, i.e. cocycles taking values in the group of invertible bounded linear operators on a Banach space, was proved  by Kalinin and Sadovskaya \cite{Kalinin2015Holonomies}.
     
     $\bullet$ \emph{Nonuniformly hyperbolic systems}. For $G=\mathbb{R}$, Katok and Hasselbatt established a nonuniform version Liv\v{s}ic theorem in their book  \cite{katok1995}. Recently, Zou and Cao \cite{ZouCao18Livsic} generalized their result to $G=GL(d,\mathbb{R})$. Backes and Poletti \cite{Backes2018A} also proved a similar result for $G=GL(d,\mathbb{R})$ later  independently.  In fact, authors for  the papers \cite{katok1995,ZouCao18Livsic,Backes2018A} only proved a nonuniform version Liv\v{s}ic theorem for ergodic  hyperbolic measures, not for general hyperbolic measures.
     
     In this paper, we prove a nonuniform version Liv\v{s}ic theorem and the H\"older regularity of  solutions for $G=GL(X)$, where X is a Banach space, and $GL(X)$ is the group of invertible bounded linear operators on X. 
\subsection{A nonuniform version Liv\v{s}ic theorem for $G=GL(X)$}
      The following closing property is used to replace the nonuniform hyperbolicity.
     
%
  Recall that a map   $C:M\to GL(X)$  is called {\em $\mu$-continuous}, if  there exists a sequence of compact set $K_n\subset M$ such that $\mu(\cup_{n\geq 1}K_n)=1$ and $C|_{K_n}$ is continuous for every $n.$ An $f$-invariant measure $\mu$ is called a {\em hyperbolic measure } if its
  Lyapunov exponents  are different from zero at $\mu$-almost every point.   
   \begin{theorem}\label{thm A} 
   	
   Let $f$ be a $C^{1+\gamma}$ diffeomorphism of a compact manifold M, preserving an hyperbolic measure $\mu$, and let $A:M\to GL(X)$ be an $\alpha$-H\"older continuous map  satisfying
   	\begin{equation}\label{1.1}
   	A(f^{n-1}p)\cdots A(fp)A(p)=Id, \quad\forall p=f^n(p),\forall n\geq 1.  
   	\end{equation}
   	Then there exists a $\mu$-continuous map $C:M\to GL(X)$ such that 
   	\[A(x)=C(fx)C(x)^{-1},\quad \text{for } \mu\text{-almost every } x\in M.  \]
   \end{theorem}
   
%
%

 We  point out that we do not assume the ergodicity of the hyperbolic measure $\mu$ in Theorem \ref{thm A}. In fact, the proof from ergodicity to  non-ergodicity is nontrivial, since the measure $\mu$ may have uncountably many ergodic components. We need to generalize a result of Fisher, Morris and Whyte \cite{Fisher04}   to overcome this problem.

\subsection{H\"older regularity of solutions} Let $f:M\to M$ be an Anosov diffeomorphism.   An $f$-invariant measure $\mu$ is called  having {\em local product structure}, if it is locally equivalent to the product of the projections of $\mu$ to the local stable and unstable manifolds. 
\begin{theorem}\label{thm B}
	Let $f$ be an Anosov diffeomorphism of a compact manifold M, $\mu$ be an  ergodic $f$-invariant measure on M with full support and local product structure, and $A:M\to GL(X)$ be an $\alpha$-H\"older continuous map. Suppose that there exists a $\mu$-continuous map $C:M\to GL(X)$ such that 
	\begin{equation}\label{1.2}
	A(x)=C(fx)C(x)^{-1},  \quad \text{for}~ \mu\text{-a.e.}~ x\in M.  
	\end{equation}
	Then C coincides $\mu$-a.e. with an $\alpha$-H\"older continuous map $\hat{C}$ satisfying the same equation everywhere. 
\end{theorem}
     This theorem gives a  positive answer to Question (\romannumeral3) for $G=GL(X)$.  The case $G=\mathbb{R}$ was first proved by Liv\v{ s}ic \cite{livsic1972}.  The case $G=GL(d,\mathbb{R})$  was proved by Sadovskaya \cite{Sado15} under the fiber bunching condition. And then Bulter \cite{Butlerconformal} improved her result by removing this additional condition.    If $X$ is a Banach space,  Sadovskaya  \cite{SadovskayaBanach} proved a similar result under the assumption that $X$ is a separable Banach space and the cocycle is fiber bunched.  We release these assumptions. And it is seen from Theorem \ref{thm A} that the $\mu$-continuity   condition  of $C$ in Theorem \ref{thm B} is natural.
    
  Since the  measure of maximal entropy or more generally the equilibrium states corresponding to H\"older continuous potentials for Anosov diffeomorphisms has full support and local product structure \cite{pollicott2001}.   As a corollary of Theorem \ref{thm A} and \ref{thm B},  we obtain immediately the uniform version Liv\v sic theorem which is proved by  Grabarnik and Guysinsky \cite{grabarnik}.

\begin{corollary}\label{cor 1.3}
	 	Let $f:M\to M$ be a transitive $C^{1+\gamma}$ Anosov diffeomorphism, and let $A:M\to GL(X)$ be an $\alpha$-H\"older continuous map  satisfying
	\begin{equation*}
	A(f^{n-1}p)\cdots A(fp)A(p)=Id, \quad\forall p=f^n(p),\forall n\geq 1.  
	\end{equation*}
	Then there exists an $\alpha$-H\"older continuous map $C:M\to GL(X)$ such that 
	\[A(x)=C(fx)C(x)^{-1},\quad \forall  x\in M.  \]
\end{corollary}
   
{\bf Acknowledgments.} We are grateful to Clark Butler and Zhiren Wang   for helpful discussions and suggestions.

\section{Preliminaries and Notations}
 \subsection{Cocycles and Exponents}
   \begin{definition}
	Suppose that $f:M\rightarrow M$ is invertible , and $A:M\rightarrow GL(X)$. A map $\mathcal{A}:M\times \mathbb{Z}\rightarrow GL(X)$ is called 	a linear  multiplicative cocycle over $f$ generated by A , if
	\[ \mathcal{A}_x^n :=\mathcal{A}(x,n)= \left\{
	\begin{array}{ll}
	A(f^{n-1}x)\cdot\cdot\cdot A(fx)A(x),& \quad\mbox{if}\ n > 0 ,\\
	Id,&\quad \mbox{if}\ n = 0,\\
	A(f^{-n}x)^{-1}\cdot\cdot\cdot A(f^{-2}x)^{-1}A(f^{-1}x)^{-1},&\quad \mbox{if}\ n <0.
	\end{array}\right.\]
	Clearly,  $ \mathcal{A}$ satisfies    $\mathcal{A}_x^{n+k} = ~\mathcal{A}_{f^kx}^{n}\circ\mathcal{A}_x^{k}$.
   \end{definition}
We introduce a metric $d$ on ~$GL(X)$ such that $\big(GL(X),d\big)$ is a  complete metric space:
\[d(A,B) = \|A - B\| + \|A^{-1} - B^{-1}\|.\]
   A cocycle $\mathcal{A}$ is called~{\em $\alpha$-H\"{o}lder} continuous, if its {\em generator}~$A=\mathcal{A}(\cdot,1):M\rightarrow GL(X)$ is~$\alpha$-H\"{o}lder continuous. 
   
    Recall that a  sequence of functions $a_n:M\to \mathbb{R}$ over a system $(M,f)$ is called a {\em subadditive cocycle} if  it satisfies
   $$a_{m+n}(x)\leq a_m(x)+a_n(f^mx)$$ for any $m,n\in \mathbb{N}$ and $x\in M.$ If $f$ is an ergodic measure preserving transformation of a probability space $(M,\mu)$ and $a_1\in L^1(M)$, then the  Subadditive Ergodic Theorem yields that there exists a set $\mathcal{R}$ of $\mu$-full measure  such that for any $x\in \mathcal{R},$  
   \[\lambda:=\lim\limits_{n\to +\infty} \frac{1}{n}\int a_nd\mu=\lim\limits_{n\to +\infty} \frac{1}{n}a_n(x).\]
   The limit $\lambda$ is called the {\em exponent} of the cocycle $a_n$ with respect to $\mu$. If we consider the continuous cocycle $\mathcal{A}:M\times \mathbb{Z}\rightarrow GL(X)$, since $\log\|\mathcal{A}_x^n\|$ and $\log\|(\mathcal{A}_x^n)^{-1}\|$ are subadditive, for  $\mu$-a.e. $x\in M,$   the limits
   \[\lambda_+(\mathcal{A},\mu):=\lim\limits_{n\to +\infty} \frac{1}{n}\int\log \|\mathcal{A}_x^n\|d\mu=\lim\limits_{n\to +\infty} \frac{1}{n}\log \|\mathcal{A}_x^n\|, \]
   and
   \begin{align*}
   \lambda_-(\mathcal{A},\mu):=-\lim\limits_{n\to +\infty} \frac{1}{n}\int\log \|(\mathcal{A}_x^n)^{-1}\|d\mu =-\lim\limits_{n\to +\infty} \frac{1}{n}\log \|(\mathcal{A}_x^{n})^{-1}\|
   \end{align*} 
   exist. $\lambda_+(\mathcal{A},\mu)$ and  $\lambda_-(\mathcal{A},\mu)$ are called the {\em upper Lyapunov exponent } and the {\em lower Lyapunov exponent } of $\mathcal{A}$ with respect to $\mu,$ respectively. 

\subsection{Hyperbolic Measure and Closing Lemma}
    Let $f$ be a $C^{1+\gamma}$  diffeomorphism of a compact manifold M.  Recall that an $f$-invariant  measure $\mu$  is said to be hyperbolic, if the Lyapunov exponents of the derivative cocycle $Df$ are non-zero for $\mu$-a.e. $x\in M.$    
    	
    We will apply the following  Katok's closing lemma \cite[Theorem S.4,13]{katok1995}(See also Theorem 15.1.2 of \cite{barreira000pesin}).
    
    \begin{lemma}[Katok's Closing Lemma]\label{fengbi}
    	Let $f\in $ {\em Dif\/f}$^{\/1+\gamma}(M)$, preserving  an ergodic hyperbolic measure $\mu$. Then there exist $\lambda>0$ and a compact set $\Lambda$ with $\mu(\Lambda)>0$, such that for any ~$\delta>0$, there exists $\beta>0$,  if $ x,f^{n}(x)\in \Lambda$ satisfying ~$d(x,f^{n}x)<\beta,$ then there exists a periodic point ~$p$ with $p=f^{n}(p)$~such that~$d(f^{i}p,f^{i}x)\leq \delta\cdot e^{-\lambda\min\{i,n-i\}}$.
    \end{lemma} 

\subsection{Anosov diffeomorphisms and fiber bunching} 
   Recall that a diffeomorphism  $f:M\to M$ is  called  Anosov, if there exists a $Df$-invariant splitting      $TM=E^s\oplus E^u$ on M, and  $\tau>0$ such that
   \[\|D_xf(v^s)\|<e^{-\tau}<1<e^\tau< \|D_xf(v^u)\| \]
   for any $x\in M$ and unit vectors $v^s\in E^s(x)$ and $v^u\in E^u(x)$. 
   \begin{definition}
   	An $\alpha$-H\"older cocycle $\mathcal{A}$ over an Anosov diffeomorphism  $f$ is called \emph{fiber bunched}, if there exists $0<\theta<1$ and $L>0$ such that for any $x\in M$ and $n\in \mathbb{N}$,
   	\[\|\mathcal{A}_x^n\|\cdot\|(\mathcal{A}_x^n)^{-1}\|\cdot e^{-\tau\alpha n}\leq L\theta^n,~~\text{and}\quad \|\mathcal{A}_x^{-n}\|\cdot\|(\mathcal{A}_x^{-n})^{-1}\|\cdot e^{-\tau\alpha n}\leq L\theta^n, \]
   \end{definition}
 The fiber bunching condition gives existence of stable and unstable holonomies of $\mathcal{A}$ on M. It's a common assumption in many theories, for instance, the continuity of Lyapunov exponents\cite{backesbrownbutler}, the existence of extremal norms\cite{BochiExtremal},  and the cohomology of cocycles\cite{Sado15}. In Theorem \ref{thm B} we do not have the fiber bunching hypothesis. Hence we introduce  a related concept. Let $D(N,\theta)$ be the set of points x satisfying 
 \begin{equation}
 \prod_{j=0}^{k-1}\|\mathcal{A}_{f^{jN}x}^{N}\|\cdot\|(\mathcal{A}_{f^{jN}x}^{N})^{-1}\|\leq e^{kN\theta}, \quad \forall k\geq 1, 
 \end{equation}
 and 
 \begin{equation}\label{2.2}
 \prod_{j=0}^{k-1}\|\mathcal{A}_{f^{-jN}x}^{-N}\|\cdot\|(\mathcal{A}_{f^{-jN}x}^{-N})^{-1}\|\leq e^{kN\theta}, \quad \forall k\geq 1.  
 \end{equation}
  It's known that $\mathcal{A}$ is fiber bunched if and only if $D(N,\theta)=M$ for some $\theta<\tau\alpha$ and $N\geq 1.$ The points in $D(N,\theta)$ for some $\theta<\tau\alpha$  also give the existence of stable and unstable holonomies.
  \begin{proposition}[\cite{Viana08},Proposition 2.5]\label{prop 2.4}
  	Given N, $\theta$ with $\theta<\tau\alpha,$ there exists $L>0$ such that for any $x\in D(N,\theta)$ and $y,z\in W^s_{loc}(x)$, the limit
  	\[H_{y,z}^s=\lim\limits_{n\to\infty}(\mathcal{A}_z^n)^{-1} \mathcal{A}_y^n\]
  	exists and satisfies $\|H_{y,z}^s-Id\|\leq L\cdot d(y,z)^\alpha$ and $H_{x,z}^s=H_{y,z}^sH_{x,y}^s$. Similarly,  for any $x\in D(N,\theta)$ and $y,z\in W^u_{loc}(x)$, the limit
  	\[H_{y,z}^u=\lim\limits_{n\to\infty}(\mathcal{A}_z^{-n})^{-1} \mathcal{A}_y^{-n}\]
  	exists and satisfies $\|H_{y,z}^u-Id\|\leq L\cdot d(y,z)^\alpha$ and $H_{x,z}^u=H_{y,z}^uH_{x,y}^u$.  Moreover, $(y,z)\mapsto H^*_{y,z}$ is continuous for $*\in\{s,u\}$, where  $y,z\in W^*_{loc}(x)$ and $x\in D(N,\theta)$.
  \end{proposition}

\section{Proof of Theorem \ref{thm A}}
    Let $f$ be a $C^{1+\gamma}$ diffeomorphism of a compact manifold M, preserving an hyperbolic measure $\mu$. Let $A:M\to GL(X)$ be an $\alpha$-H\"older continuous map satisfying \eqref{1.1}.   We may first assume that $\mu$ is an ergodic  measure. The general case will be considered in the subsection 3.2.
   \subsection{The case $\mu$ is ergodic.}
     Since the cocycle $\mathcal{A}$ generated by $A$ is  continuous,   the Subadditive Ergodic Theorem yields that there exists a set $\mathcal{R}$ of $\mu$-full measure  such that for any $x\in \mathcal{R},$  
     \[\lambda_+(\mathcal{A},\mu)=\lim\limits_{n\to +\infty} \frac{1}{n}\int\log \|\mathcal{A}_x^n\|d\mu=\lim\limits_{n\to +\infty} \frac{1}{n}\log \|\mathcal{A}_x^n\|, \]
     and
     \begin{align*}
      -\lambda_-(\mathcal{A},\mu)=\lim\limits_{n\to +\infty} \frac{1}{n}\int\log \|(\mathcal{A}_x^n)^{-1}\|d\mu & =\lim\limits_{n\to +\infty} \frac{1}{n}\int\log \|\mathcal{A}_x^{-n}\|d\mu\\
      &=\lim\limits_{n\to +\infty} \frac{1}{n}\log \|\mathcal{A}_x^{-n}\|. 
     \end{align*} 
     By  \cite[Theorem 1.5]{Klinin18},  the upper Lyapunov exponent $\lambda_+(\mathcal{A},\mu)$ and lower Lyapunov exponent $\lambda_-(\mathcal{A},\mu)$ can be approximated in terms of the norms of its periodic date, that is, for any $\varepsilon>0$,  there exists a periodic point $p=f^n(p)$ such that 
     \[\big|\lambda_+(\mathcal{A},\mu)-\frac{1}{n}\log \|\mathcal{A}_p^n\|\big|<\varepsilon, \quad \big|\lambda_-(\mathcal{A},\mu)-\frac{1}{n}\log \|(\mathcal{A}_p^n)^{-1}\|^{-1}\big|<\varepsilon.\]
    Thus \eqref{1.1} implies $\lambda_+(\mathcal{A},\mu)=\lambda_-(\mathcal{A},\mu)=0.$ Then for a fixed $\varepsilon>0$ and any point $x\in \mathcal{R}$, we define the \emph{Lyapunov norm}  $\|\cdot\|_x=\|\cdot\|_{x,\varepsilon}$ in $X$  as follows:
    \[\|u\|_x:=\sum_{n=-\infty}^{+\infty}\|\mathcal{A}_x^n(u)\|e^{-\varepsilon|n|},\quad \forall u\in X. \]
    By \cite[Proposition 3.1]{Kalinin16}, the Lyapunov norm satisfies the following properties:
    \begin{enumerate}
    	\item  For any $x\in \mathcal{R},$ 
    	 \begin{equation}\label{3.1}
    	 e^{-\varepsilon}\|u\|_x\leq\|A(x)u \|_{fx}\leq e^\varepsilon\|u\|_x,\quad \forall u\in X.
    	 \end{equation}
    	\item    There exists an $f$-invariant subset $\mathcal{R}_\varepsilon\subset \mathcal{R}$ with $\mu(\mathcal{R}_\varepsilon)=1$ and a measurable function $K_\varepsilon(x)$ such that for any $x\in \mathcal{R}_\varepsilon$,
    	   \begin{equation}\label{3.2}
    	    \|u\|\leq \|u\|_x \leq K_\varepsilon(x)\|u\|,\quad \forall u\in X,~ \text{and}
    	   \end{equation}
    	   \begin{equation}\label{3.3}
    	    K_\varepsilon(x)e^{-\varepsilon}\leq K_\varepsilon(fx)\leq K_\varepsilon(x)e^\varepsilon.
    	   \end{equation}
    \end{enumerate}
 
     For any $l\geq 1,$ we define 
     \begin{equation}\label{3.4}
       \mathcal{R}_{\varepsilon,l}=\{x\in \mathcal{R}_\varepsilon: K_\varepsilon(x)\leq l  \}.
     \end{equation}
      Then $\mu( \mathcal{R}_{\varepsilon,l})\to 1$ as $n\to \infty$. Without loss of generality, we may assume $\mathcal{R}_{\varepsilon,l}$ is a compact set by using the Lusin's theorem.
      
      We use the Lyapunov norm to estimate the norm of  $\mathcal{A}$ along an orbit segment that is  close to  a regular one. Denote $x_i=f^i(x), y_i=f^i(y).$  Let $\varepsilon_{0}=\frac{1}{4}\lambda\alpha.$
\begin{lemma}\label{lemma 3.1}
      	Let $f,A,\mu$ be as above.  Then for any $l>1,~0<\varepsilon<\varepsilon_{0}$, there exist~$ \delta_1,c_1>0$, such that for any~$x,f^{n}(x)\in \mathcal{R}_{\varepsilon,l},~y\in M, 0<\delta<\delta_1$ satisfying  $d(f^{i}(x),f^{i}(y))\leq \delta e^{-\lambda\min\{i,n-i\}},i=0,\cdots,n$, we have:
      	$$c_1^{-1}e^{-2\varepsilon i}\leq m(\mathcal{A}_y^i)\leq\|\mathcal{A}_y^i\|\leq c_1e^{2\varepsilon i},$$
      	$$c_1^{-1}e^{-2\varepsilon (n-i)}\leq m(\mathcal{A}_{y_i}^{n-i})\leq\|\mathcal{A}_{y_i}^{n-i}\|\leq c_1e^{2\varepsilon(n-i)},$$
      	where $m(B):=\inf\limits_{\|v\|=1}\|Bv\|=\|B^{-1}\|^{-1}$.
\end{lemma}
  \begin{proof}
      We only prove $c_1^{-1}e^{-2\varepsilon i}\leq m(\mathcal{A}_y^i)\leq\|\mathcal{A}_y^i\|\leq c_1e^{2\varepsilon i},$ the other one can be obtained in a similar fashion.
      
      For any $0\leq j\leq n-1$, $u\in X$ and $0<\varepsilon<\varepsilon_{0}$, by \eqref{3.1} and \eqref{3.2}, 
      \begin{align*}
      \|A(y_j)u \|_{x_{j+1}}= & ~\|\big(A(y_j)-A(x_j)+A(x_j)\big)u \|_{x_{j+1}}  \\
      \geq & ~\|A(x_j)u \|_{x_{j+1}}-\|\big(A(y_j)-A(x_j)\big)u  \|_{x_{j+1}} \\
      \geq  & ~e^{-\varepsilon}\|u\|_{x_j}-K_{\varepsilon}(x_{j+1})\|A(y_j)-A(x_j) \|\cdot\|u\|_{x_{j}}.
      \end{align*}
      Since $x,f^{n}x\in \mathcal{R}_{\varepsilon,l}$,  it follows from \eqref{3.3} that
      $$K_{\varepsilon}(x_{j+1}) \leq le^{\varepsilon \min\{j+1,n-j-1\}}\leq le^\varepsilon\cdot e^{\varepsilon \min\{j,n-j\}}.$$
      Note that $A(x)$ is $\alpha$-H\"{o}lder continuous, one has
      \[\|A(y_{j})-A(x_{j})\|\leq c_0\cdot d(y_{j},x_{j})^{\alpha}\leq c_0\delta^\alpha e^{-\lambda\alpha \min\{j,n-j\}}. \]	
      Hence for any $0\leq j\leq n-1$,  $u\in X$,
      \begin{equation*}
      K_{\varepsilon}(x_{j+1})\|A(y_j)-A(x_j) \|\cdot\|u\|_{x_{j}}
      \leq  c_0le^\varepsilon \delta^\alpha e^{(\varepsilon-\lambda\alpha )\min\{j,n-j\}}\|u\|_{x_j}.
      \end{equation*}
      Therefore,  the previous inequality gives    
      \begin{align*}
      \|A(y_j)u\|_{x_{j+1}}  \geq   \big(e^{-\varepsilon}-c_0le^\varepsilon \delta^\alpha e^{(\varepsilon-\lambda\alpha )\min\{j,n-j\}}\big)\|u\|_{x_j}. 
      \end{align*}
       Similarly, 
      \begin{equation*}
      \|A(y_j)u\|_{x_{j+1}}\leq  \big(e^{\varepsilon}+c_0le^\varepsilon \delta^\alpha e^{(\varepsilon-\lambda\alpha )\min\{j,n-j\}}\big)\|u\|_{x_j}. 
      \end{equation*}  
      Thus for any $v\in X$, we conclude
      \begin{equation*}\begin{split}
      \|\mathcal{A}_{y}^{i}(v)\| & \geq K_\varepsilon(x_i)^{-1}\|\mathcal{A}_{y}^{i}(v)\|_{x_i}\\
      & \geq~l^{-1}e^{-\varepsilon i}\prod\limits_{j=0}^{i-1}\big(e^{-\varepsilon}-c_0le^\varepsilon \delta^\alpha e^{(\varepsilon-\lambda\alpha     )\min\{j,n-j\}}\big)\|v\|_{x_0}\\
      & \geq l^{-1}e^{-2\varepsilon i}\prod\limits_{j=0}^{i-1}\big(1-c_0le^{2\varepsilon} \delta^\alpha e^{(\varepsilon-\lambda\alpha     )\min\{j,n-j\}}\big)\|v\|.                            
      \end{split}	\end{equation*}
      Take $\delta_1>0$ small enough such that $1-2c_0le^{2\varepsilon} \delta_1^\alpha>0.$ Then for any $0<\delta<\delta_1$, since $\varepsilon-\lambda \alpha<0$, we can estimate 
      \begin{align*}
      \sum\limits_{j=0}^{n-1}\log\big(1-c_0le^{2\varepsilon}\delta^\alpha e^{(\varepsilon-\lambda\alpha)\min\{j,n-j\}}\big)
      \geq & ~ -\sum\limits_{j=0}^{n-1} 2c_0le^{2\varepsilon}\delta^\alpha e^{(\varepsilon-\lambda\alpha)\min\{j,n-j\}}\\
      \geq &~ -\tilde{c}_0,
      \end{align*}
      where $\tilde{c}_0=\tilde{c}_0(l,\varepsilon)$ is a constant. It follows that
       \begin{equation*}\begin{split}
      m(\mathcal{A}_{y}^{i})=\inf\limits_{\|v\|=1}\|\mathcal{A}_{y}^{i}(v)\|  & \geq l^{-1}e^{-2\varepsilon i}\prod\limits_{j=i}^{n-1}\big(1-c_0le^{2\varepsilon} \delta^\alpha e^{(\varepsilon-\lambda\alpha     )\min\{j,n-j\}}\big)\\
      & \geq   l^{-1} e^{-\tilde{c}_0} e^{-2\varepsilon i}\\
      & =:c_1^{-1}e^{-2\varepsilon i}.                         
      \end{split}	\end{equation*}
      
      Similarly, we can also obtain  
      \begin{equation*}\begin{split}
      \|\mathcal{A}_{y}^{i}(v)\|\leq &  ~\|\mathcal{A}_{y}^{i}(v)\|_{x_i}= ~\|A(y_{i-1})\cdots A(y)v \|_{x_i}\\
      \leq &~\prod\limits_{j=0}^{i-1}\big(e^{\varepsilon}+c_0le^\varepsilon\delta^\alpha e^{(\varepsilon-\lambda\alpha     )\min\{j,n-j\}}\big)\|v\|_{x_0}\\
      \leq &~le^{\varepsilon i}\prod\limits_{j=0}^{i-1}\big(1+c_0l\delta^\alpha e^{(\varepsilon-\lambda\alpha)\min\{j,n-j\}}\big)\|v\|\\
      \leq & ~ c_1e^{\varepsilon i}\|v\|,                         
      \end{split}	\end{equation*}
     which implies $\|\mathcal{A}_{y}^{i}\|\leq c_1e^{\varepsilon i}$. This finishes the proof.      
  \end{proof}
      For a fixed $\varepsilon<\varepsilon_0$, denote~$G_{l}=\Lambda\cap\mathcal{R}_{\varepsilon,l},$ where $\Lambda$ is given by Lemma \ref{fengbi}.  Then for~$l$ large enough, we have~$\mu(G_{l})>0.$ Let $~G^{\prime}_{l}=\text{supp}(\mu_{G_{l}})$, where ~$\mu_{G_{l}}$ is defined by:~
      $\mu_{G_{l}}(B):=\mu(B\cap G_{l})/\mu(G_{l}).$
      
      We need the following lemma to find  a dense orbit in $G^{\prime}_{l}$.
        \begin{lemma}[\cite{ZouCao18Livsic} Lemma 4.1]\label{lemma 3.2}
        	Let $f$ be a continuous map of a compact metric space $M$, preserving an ergodic measure $\mu$. Suppose $D$ is a closed set with $\mu(D)>0$, and $E=\{x\in \text{\em{supp}}(\mu_{D}):\overline{\mathcal{O}(x)\cap \text{\em{supp}}(\mu_{D})}=\text{\em{supp}}(\mu_{D})\}$, where ~$\mu_{D}$ is defined by
        	$\mu_{D}(B):=\mu(B\cap D)/\mu(D)$. Then ~$\mu(E)=\mu(\text{\em{supp}}(\mu_{D}))=\mu(D)$.
        \end{lemma}
      	Since~$G_{l}$ is closed,  we have$~G^{\prime}_{l}\subset G_{l}$. By Lemma \ref{lemma 3.2}, one can find a point~$z\in G^{\prime}_{l}\subset\Lambda\cap\mathcal{R}_{\varepsilon,l}$, such that~$\overline{\mathcal {O}(z)\cap G^{\prime}_{l}}=G^{\prime}_{l}.$  We define the map $C:\mathcal {O}(z)\cap G^{\prime}_{l}\to GL(X)$  by $C(f^nz)=\mathcal{A}_z^n,$ for any $f^nz\in G^{\prime}_{l}$. 
      	We shall prove that~$C$ is uniformly continuous on~$\mathcal {O}(z)\cap G^{\prime}_{l}$, so that $C$ can be extended to $~G^{\prime}_{l}$. 
      	
\begin{lemma}\label{lemma 3.3}
	Given any~$l>1,~0<\varepsilon<\varepsilon_{0}$, there exists~$0<\delta_2<1,c_2>0$, such that for any $0<\delta<\delta_2$, there exists $\beta>0,$ if $x,f^{n}x\in G_{l}$ satisfying~$d(x,f^{n}x)\leq \beta$, then $d(\mathcal{A}_x^n,Id)\leq c_2 \delta^\alpha.$ 	
\end{lemma} 

   \begin{proof}
   	Since   $f$ has the closing property on $G_{l}\subset \Lambda$ , there exists $\lambda>0$, such that for any $0<\delta<1$, there exists $\beta>0,$ such that if  $x,f^nx\in G_{l}$ with $d(x,f^nx)<\beta$, then one can find a periodic point $p=f^np\in M$, such that $$d(f^{i}p,f^{i}x)\leq \delta\cdot e^{-\lambda\min\{i,n-i\}},~\quad \forall~ i=0,\cdots,n.$$
   	We estimate $\|\mathcal{A}_x^n-Id\|$ first.
   	\begin{align*}
   	 \mathcal{A}_x^n-\mathcal{A}_p^n
   	= & ~\mathcal{A}_{x_1}^{n-1}\circ(A(x_0)-A(p_0))+(\mathcal{A}_{x_1}^{n-1}-\mathcal{A}_{p_1}^{n-1})\circ A(p_0)\\
   	= & ~\mathcal{A}_{x_1}^{n-1}\circ(A(x_0)-A(p_0))+ \mathcal{A}_{x_2}^{n-2}\circ(A(x_1)-A(p_1))\circ A(p_0)+\\
   	& (\mathcal{A}_{x_2}^{n-2}-\mathcal{A}_{p_2}^{n-2})\circ\mathcal{A}_p^2)\\
   	= & ~\cdots=\sum_{i=0}^{n-1}\mathcal{A}_{x_{i+1}}^{n-i-1}\circ(A(x_i)-A(p_i))\circ\mathcal{A}_p^i.
   	\end{align*}
   	Since $\mathcal{A}_p^n=Id$, by Lemma \ref{lemma 3.1} , for any $0\leq i\leq n$, 
   	$$\|\mathcal{A}_p^i\|=\|(\mathcal{A}_{p_{i}}^{n-i})^{-1}\|\leq c_1e^{2\varepsilon \min\{i,n-i\}}.$$
   	Note that $\|A(x_i)-A(p_i)\|\leq c_0\delta^\alpha e^{-\lambda\alpha \min\{i,n-i\}}.$ Denote $m=\left\lfloor \frac{n}{2}\right\rfloor,$ then Lemma \ref{lemma 3.1} and  the fact $\varepsilon<\varepsilon_{0}$ give 
   	\begin{align*}
   	& \sum_{i=0}^{m}\|\mathcal{A}_{x_{i+1}}^{n-i-1}\|\cdot\|A(x_i)-A(p_i)\|\cdot\|\mathcal{A}_p^i\|\\
   	\leq & ~ \sum_{i=0}^{m}\|\mathcal{A}_x^n\|\cdot\|(\mathcal{A}_x^{i+1})^{-1}\| \cdot\|A(x_i)-A(p_i)\|\cdot\|\mathcal{A}_p^i\|\\
   	\leq &~ \|\mathcal{A}_x^n\|\cdot\sum_{i=0}^{m}c_1e^{2\varepsilon (i+1)}\cdot c_0\delta^\alpha e^{-\lambda\alpha i}\cdot c_1e^{2\varepsilon i}\\
   	\leq &~ \tilde{c}_1\delta^\alpha\cdot\|\mathcal{A}_x^n\|,
   	\end{align*}
   	and
   	\begin{align*}
   	& \sum_{i=m+1}^{n}\|\mathcal{A}_{x_{i+1}}^{n-i-1}\|\cdot\|A(x_i)-A(p_i)\|\cdot\|\mathcal{A}_p^i\|\\
   	\leq &~ \sum_{i=m+1}^{n}c_1e^{2\varepsilon (n-i-1)}\cdot c_0\delta^\alpha e^{-\lambda\alpha (n-i)}\cdot c_1e^{2\varepsilon (n-i)}\\
   	\leq &~ \tilde{c}_1\delta^\alpha,
   	\end{align*}
   	where $\tilde{c}_1$ is a constant. Therefore, 
   	\begin{equation}\label{3.5}
   	\begin{split}
   	\|\mathcal{A}_x^n\|-1 & \leq \|\mathcal{A}_x^n-\mathcal{A}_p^n\| \\
   	& \leq \sum_{i=0}^{n-1}\|\mathcal{A}_{x_{i+1}}^{n-i-1}\|\cdot\|A(x_i)-A(p_i)\|\cdot\|\mathcal{A}_p^i\|\\
   	& \leq  \tilde{c}_1\delta^\alpha\|\mathcal{A}_x^n\|+\tilde{c}_1\delta^\alpha.    
   	\end{split}   
   	\end{equation}
   	Take $\delta_2$ small enough such that $\tilde{c}_1\delta_2^\alpha<\frac{1}{5}.$ Then for any $0<\delta<\delta_2,$ we obtain 
   	$$\|\mathcal{A}_x^n\|\leq \frac{1+\tilde{c}_1\delta_2^\alpha}{1-\tilde{c}_1\delta_2^\alpha}\leq \frac{3}{2}.$$
   	Thus equation \eqref{3.5} gives
   	\begin{align*}
   	\|\mathcal{A}_x^n-Id\|& =\|\mathcal{A}_x^n-\mathcal{A}_p^n\| \\
   	& \leq  \tilde{c}_1\delta^\alpha\|\mathcal{A}_x^n\|+\tilde{c}_1\delta^\alpha\\
   	& \leq   \frac{5}{2}\tilde{c}_1\delta^\alpha.   
   	\end{align*}
   
   	To estimate $\|(\mathcal{A}_x^n)^{-1}-Id\|.$ Let $Y=Id-\mathcal{A}_x^n$, then
   	$$(\mathcal{A}_x^n)^{-1}=(Id-Y)^{-1}=Id+Y+Y^2+\cdots.$$
   Since $\|Y\|=\|Id-\mathcal{A}_x^n\|\leq \frac{5}{2}\tilde{c}_1\delta_2^\alpha \leq1/2,$ we have 
   	\begin{align*}
   	\|(\mathcal{A}_x^n)^{-1}-Id\|  \leq \sum_{i=1}^{\infty}\|Y^i\| 
   	\leq  \sum_{i=1}^{\infty}(\frac{5}{2}\tilde{c}_1\delta^\alpha)^i
   	\leq  5\tilde{c}_1\delta^\alpha.
   	\end{align*}
   	Therefore, $d(\mathcal{A}_x^n,Id)=\|\mathcal{A}_x^n-Id\|+\|(\mathcal{A}_x^n)^{-1}-Id\|\leq \frac{15}{2}\tilde{c}_1\delta^\alpha=:c_2\delta^\alpha.$ This completes the proof.  
   \end{proof}     
   
   We now prove the uniform continuity of  $C$ on~$\mathcal {O}(z)\cap G^{\prime}_{l}$, that is, for any~$\delta>0,$ there exist~$\beta,c_3>0,$ such that for any~$x,f^nx\in \mathcal {O}(z)\cap G^{\prime}_{l}$ satisfying $d(x,f^nx)<\beta,$ one has~$d(C(x),C(f^nx))\leq c_3\delta^\alpha.$ Suppose $x=f^k(z)$, then
    \begin{equation}\label{3.6}
    C(f^nx)C(x)^{-1}=\mathcal{A}_z^{n+k}(\mathcal{A}_z^k)^{-1}=\mathcal{A}_x^n.
    \end{equation} Therefore,
   \begin{align*}
   ~d(C(x),C(f^nx)) & =\|C(x)-C(f^nx)\|+\|C(x)^{-1}-C(f^nx)^{-1}\|\\
   & \leq \|Id-\mathcal{A}_x^n\|\cdot\|C(x)\|+\|C(x)^{-1}\|\cdot\|Id-(\mathcal{A}_x^n)^{-1}\|\\
   &\leq (\|C(x)\|+\|C(x)^{-1}\|)\cdot c_2\delta^\alpha.
   \end{align*} 
   It's enough to prove that~$\|C(x)\|$ and $\|C(x)^{-1}\|$ are uniformly bounded on~$\mathcal {O}(z)\cap G^{\prime}_{l}$.   
   By Lemma \ref{lemma 3.3}, for any $\delta<\delta_2,$ there exists $\beta>0$ such that for any~$x,f^nx\in \mathcal {O}(z)\cap G^{\prime}_{l}$ with $d(x,f^nx)<\beta,$ we obtain $ d(\mathcal{A}_x^n,Id)\leq c_2\delta^\alpha.$
   Since~$\mathcal {O}(z)\cap G^{\prime}_{l}$ is dense in~$G^{\prime}_{l}$, we may choose a segment~$\mathcal{O}_L=\{f^kz\}_{k\in [-L,L]}$ such that~$\mathcal{O}_L\cap G^{\prime}_{l}$ forms a $\beta$-net of $G^{\prime}_{l}$. Then for any~$f^mz\in \mathcal{O}(z)\cap G^{\prime}_{l}$, there exists $f^kz\in\mathcal{O}_L\cap G^{\prime}_{l}$ with $k\in[-L,L]$ such that $d(f^kz,f^mz)<\beta$. Then we have~$d(\mathcal{A}_{f^kz}^{m-k},Id)\leq c_2\delta^\alpha,$ which implies $\|\mathcal{A}_{f^kz}^{m-k}\|,\|(\mathcal{A}_{f^kz}^{m-k})^{-1}\|\leq 1+c_2\delta_2^\alpha.$  Let $\tilde{c}_3=\max_{i\in[-L,L]}\{\|\mathcal{A}_z^i\|,\|(\mathcal{A}_z^i)^{-1}\|\}$, then the equality~$\mathcal{A}_z^m=\mathcal{A}_{f^kz}^{m-k}\mathcal{A}_z^k$ gives $\|\mathcal{A}_z^m\|,\|(\mathcal{A}_z^m)^{-1}\|\leq \tilde{c}_3(1+c_2\delta_2^\alpha)$, that is, $\|C(x)\|,\|C(x)^{-1}\|$ are uniformly bounded on~$\mathcal {O}(z)\cap G^{\prime}_{l}$.  	Therefore,  $C$ can be extended continuously to~$G^{\prime}_{l}$.  Then by \eqref{3.6},
   \begin{equation}\label{3.7}
      \mathcal{A}_x^n=C(f^nx)C(x)^{-1},\quad \forall x,f^n(x)\in G^{\prime}_{l}.
   \end{equation}  	
   	
   	   Now  we extend $C$ to~$\bigcup\limits_{i=0}^{\infty}f^i(G^{\prime}_{l})$ as follows: If
   $y\in \left(\bigcup\limits_{i=0}^{n}f^i(G^{\prime}_{l})\right)\setminus \bigcup\limits_{i=0}^{n-1}f^i(G^{\prime}_{l})$, then define~$C(y):=\mathcal{A}_{f^{-n}(y)}^nC(f^{-n}y)$.   To show  the map C satisfies \eqref{1.2},  for any $y\in\bigcup\limits_{i=0}^{\infty}f^i(G^{\prime}_{l})$, we may assume
   $y\in \left(\bigcup\limits_{i=0}^{n}f^i(G^{\prime}_{l})\right)\setminus \bigcup\limits_{i=0}^{n-1}f^i(G^{\prime}_{l})$ for some $n\geq 0$, where $\bigcup\limits_{i=0}^{-1}f^i(G^{\prime}_{l})$ denotes the empty set.  Then we have $f^{-n}(y)\in G^{\prime}_{l}$,     and
    \begin{align*}
    f(y)\in \left(\bigcup\limits_{i=1}^{n+1}f^i(G^{\prime}_{l})\right)\setminus \bigcup\limits_{i=1}^{n}f^i(G^{\prime}_{l})
    &\subset \left(\bigcup\limits_{i=0}^{n+1}f^i(G^{\prime}_{l})\right)\setminus \bigcup\limits_{i=1}^{n}f^i(G^{\prime}_{l})\\
    &\subset \left(\left(\bigcup\limits_{i=0}^{n+1}f^i(G^{\prime}_{l})\right)\setminus \bigcup\limits_{i=0}^{n}f^i(G^{\prime}_{l})\right)\bigcup G^{\prime}_{l}.
    \end{align*}
    If $f(y)\in \left(\bigcup\limits_{i=0}^{n+1}f^i(G^{\prime}_{l})\right)\setminus \bigcup\limits_{i=0}^{n}f^i(G^{\prime}_{l})$, then by the definition, $$C(fy)=\mathcal{A}_{f^{-n}(y)}^{n+1}C(f^{-n}y)=\mathcal{A}_{f^{-n}(y)}^{n+1}(\mathcal{A}_{f^{-n}(y)}^n)^{-1}C(y)=A(y)C(y).$$
    If $f(y)\in G^{\prime}_{l},$ then by $f^{-n}(y)\in G^{\prime}_{l}$,  \eqref{3.7} and the definition of $C(y)$, we obtain
    \[C(fy)=\mathcal{A}_{f^{-n}(y)}^{n+1}C(f^{-n}y)=A(y)\mathcal{A}_{f^{-n}(y)}^{n}C(f^{-n}y)=A(y)C(y).  \]
     Since~$\mu\left(\bigcup\limits_{i=0}^{\infty}f^i(G^{\prime}_{l})\right)=1$, we conclude that the map $C$ is defined almost everywhere and satisfying
   $$A(x)=C(fx)C(x)^{-1},\quad \text{ for } \mu\text{-a.e. } x\in M.$$ 	
   
  At last, it's left to prove C is $\mu$-continuous. Note that  $GL(X)$ is not separable in general. We can not use the   Lusin's Theorem directly.  However,  we shall show that the image of C on a  set of $\mu$-full measure is contained in a separable complete  subspace of  $GL(X).$ Let
  \begin{equation}\label{3.8}
  K:=\overline{\bigcup\limits_{n\in \mathbb{Z}}\left\{\mathcal{A}_x^n:x\in M\right\}}\subset GL(X). 
  \end{equation}
  Since $x\mapsto \mathcal{A}_x^n$ is continuous, $\left\{\mathcal{A}_x^n:x\in M\right\}$ is compact for any $n\in\mathbb{Z}.$ Therefore $\bigcup\limits_{n\in \mathbb{Z}}\left\{\mathcal{A}_x^n:x\in M\right\}$ is a $\sigma$-compact subset and which implies K is separable. We claim that  the image of C on $\bigcup\limits_{i=0}^{\infty}f^i(G^{\prime}_{l})$ is contained in K. Indeed,
   for any $y\in\left(\bigcup\limits_{i=0}^{n}f^i(G^{\prime}_{l})\right)\setminus \bigcup\limits_{i=0}^{n-1}f^i(G^{\prime}_{l}),$ since $f^{-n}(y)\in G^{\prime}_{l}$, there exists a sequence $\{n_k\}_{k\geq 1}$ such that $f^{n_k}(z)\to f^{-n}(y)$ as $k\to \infty.$ By  the continuity of C on $G^{\prime}_{l},$  $C(f^{-n}y)=\lim\limits_{k\to \infty}C(f^{n_k}z)=\lim\limits_{k\to \infty}\mathcal{A}_z^{n_k}$, which implies 
  \[C(y)=\mathcal{A}_{f^{-n}(y)}^nC(f^{-n}y)=\lim\limits_{k\to\infty } \mathcal{A}_{f^{n_k}(z)}^n\mathcal{A}_z^{n_k}=\lim\limits_{k\to\infty } \mathcal{A}_z^{n_k+n}\in K.\]
  Therefore $C:\bigcup\limits_{i=0}^{\infty}f^i(G^{\prime}_{l})\to K.$ Since $\bigcup\limits_{i=0}^{\infty}f^i(G^{\prime}_{l})$ is of $\mu$-full measure and K is a separable complete  metric space, by Lusin's Theorem,  for any $k\geq 1,$ there exists a compact subset $F_k\subset\bigcup\limits_{i=0}^{\infty}f^i(G^{\prime}_{l})$ with $\mu(F_k)>1-\frac{1}{n}$ such that $C|_{F_k}$ is continuous. This shows the $\mu$-continuity of $C.$
   

 
 \subsection{The general case}
 %
 
  If $\mu$ is not ergodic, since almost every ergodic component of $\mu$ is an ergodic hyperbolic measure,  by subsection 3.1,   we  reduce the proof of Theorem \ref{thm A} to proving the following  Theorem \ref{thm 3.4}.
  
 \begin{theorem}\label{thm 3.4}
 	Let  $f:M\to M$ be a homeomorphism, $\mu$ be an  $f$-invariant Borel probability measure, and let $A:M\to GL(X)$ be continuous.  If for almost every ergodic component $\nu$ of  $\mu$, there exists a $\nu$-continuous map $C_{\nu}:M\to GL(X)$ such that $A(x)=C_{\nu}(fx)C_{\nu}(x)^{-1}$ holds for  $\nu$-a.e. $x\in M$. Then there exists a $\mu$-continuous map $C:M\to GL(X)$, such that
 	\[A(x)=C(fx)C(x)^{-1},  \quad\text{for~} \mu\mbox{-a.e. } x\in M. \] 
 \end{theorem}
  This theorem generalizes   Theorem 3.4 in \cite{Fisher04} in the sense that $GL(X)$ is neither separable nor locally compact. In fact, they proved that if $G$ is a  second countable and  locally compact group, $A:M\to G$ and $C_{\nu}:M\to G$ are measurable for almost every ergodic component $\nu$ of  $\mu$ such that $A(x)=C_{\nu}(fx)C_{\nu}(x)^{-1}$ for  $\nu$-a.e. $x\in M$.   Then there exists a measurable map $C:M\to G$ such that $A(x)=C(fx)C(x)^{-1}$ for $\mu$-a.e. $x\in M.$  In our setting,  $GL(X)$ is not separable. We shall prove  that the image of $A$  and $C_\nu$ are in a separable subset K (but not a subgroup) of  $GL(X)$ almost everywhere.  And thanks to the metric of $GL(X)$,  the local compactness condition is not needed in our setting.  
  
  \begin{proof}[Proof of Theorem \ref{thm 3.4}]
  	For any given ergodic component $\nu$ of  $\mu$, we shall prove first that 
  	$C_{\nu}:M\to K$, where K is given by \eqref{3.8}. 
  	
  	Since $C_\nu$ is $\nu$-continuous, there exists a sequence of compact subsets $F_1\subset F_2\subset\cdots$ of M with  $\nu(\cup_{n\geq 1}F_n)=1$ such that $C_\nu|_{F_n}$ is continuous for every $n\geq1.$   Let $D_n=\mbox{supp}(\nu_{F_n})$, and let $E_n=\{x\in D_n:\overline{ \mathcal{O}(x)\cap D_n}=D_n \}$, ~ where $\nu_{F_n}$ is defined by
  	$\nu_{F_n}(B):=\nu(B\cap F_n)/\nu(F_n)$. Then by Lemma \ref{lemma 3.2}, $\nu(E_n)=\nu(F_n)$ for every $n\geq 1.$ Since the sequence $\{F_n\}_{n\geq 1}$ is  increasing,  we have $\nu(\cap_{n\geq 1}E_n)=\nu(F_1)>0.$ Choose $z\in \cap_{n\geq 1}E_n$. Then 
  	$$\overline{ \mathcal{O}(z)\cap D_n}=D_n,\quad \forall n\geq 1.$$
  	Without loss of generality, we may assume $C_\nu(z)=Id.$ Indeed, if $C_\nu(z)\neq Id,$ we may replace the map $C_\nu(x)$ by the map $\widetilde{C}_\nu(x):=C_\nu(x)C_\nu(z)^{-1}.$  Then  $\widetilde{C}_\nu(z)=Id$ and   $A(x)=C_{\nu}(fx)C_\nu(z)^{-1}C_\nu(z)C_{\nu}(x)^{-1}=\widetilde{C}_\nu(fx)\widetilde{C}_\nu(x)^{-1}$ for $\nu$-a.e. $x\in M$.
  	
  	Now for any $f^k(z)\in D_n$, we have $C_\nu(f^kz)=\mathcal{A}_z^kC_\nu(z)=\mathcal{A}_z^k.$  Since $\mathcal{O}(z)$ is dense in $D_n$, and $C_\nu$ is continuous on $D_n$,  we have 
  	$$C_\nu(y)\in K,~~\forall y\in D_n,n\geq 1.$$
  	It follows from $\nu(\cup_{n\geq 1}D_n)=1$ that 	$C_\nu(y)\in K$ for $\nu$-a.e. $y\in M$,
  
 To continue the proof Theorem \ref{thm 3.4}, we shall make some  reductions.
We begin by introducing  an ergodic decomposition theorem, a theorem of Rohlin and some lemmas of measurability.  

For a given  measurable partition  $\mathcal{P}$ of $M$, denote $\Omega=M/\mathcal{P}$.   Let $\pi:M\to \Omega$ be the corresponding projection, and let	$\hat{\mu}=\pi_*\mu$.	
We state an ergodic decomposition theorem from the point of  Rohlin's view. 
\begin{theorem}[Theorem 5.1.3 \cite{VianaOliveira16}]\label{thm 3.5}
	Let $f,\mu$ be as above. Then there exists a measurable subset $M_0\subset M$ with $\mu(M_0)=1$, a partition $\mathcal{P}$ of $M_0$ into measurable  subsets and a family  $\{\mu_\omega:\omega\in\Omega=M_0/\mathcal{P}\}$ of probability measures on M such that
	\begin{itemize}
		\item   $\mu_\omega(\pi^{-1}(\omega))=1$ for $\hat{\mu}$-a.e. $\omega\in\Omega$, 
		\item   $\mu_\omega$ is $f$-invariant and ergodic for $\hat{\mu}$-a.e. $\omega\in\Omega$, 
		\item   $\omega\mapsto \mu_\omega(E)$ is measurable, for every measurable subset $E\subset M$, and
		\[\mu(E)=\int \mu_\omega(E) d \hat{\mu}(\omega).\]
	\end{itemize}
\end{theorem}

 \begin{definition}	
	A Borel probability space $(S,m)$ is called \emph{standard}, if S is Borel isomorphic to a complete separated metric space, and m is a Borel probability measure on S.
\end{definition}
We simplify the problem by using  the following Rohlin's theorem \cite{Rohlin}. This  statement comes from \cite[Proposition 2.21]{Fisher04}. 
\begin{proposition}[Rohlin \cite{Rohlin}]\label{prop 3.7}
	Assume the notation of Theorem \ref{thm 3.5}. Then there is a partition of $\Omega$ into countably many Borel subsets $\Omega_1,\Omega_2,\cdots$, such that for any $k\geq 1$, there exists a standard Borel probability space $(S_k,m_k)$ and a isomorphism $\theta_k:(\pi^{-1}(\Omega_k),\mu)\to (\Omega_k\times S_k,\hat{\mu}\times m_k).$ 
\end{proposition}	
\begin{proof}[Sketch of  the proof]   
	Two standard Borel probability spaces $(S_1,\mu_1)$ and $(S_2\mu_2)$ are called {\em of the same type} if   $S_1$ is isomorphic Mod 0 to $S_2$. Since any standard Borel probability space is isomorphic Mod 0 to  the space consisting of an interval with ordinary Lebesgue measure and a sequence of points with positive measure\cite[ No. 4, \S 2]{Rohlin}, there are only countably many equivalent classes. Then by \cite[ Theorem (\uppercase\expandafter{\romannumeral1}) in \S 4]{Rohlin}, there is a measurable decomposition $\Omega=\Omega_1\cup\Omega_2\cup\cdots$ such that for any $\omega_1,\omega_2\in \Omega_k$, $(\pi^{-1}(\omega_1),\mu_{\omega_1} )$ and $(\pi^{-1}(\omega_2),\mu_{\omega_2} )$ are of the same type.  We conclude by using \cite[ Theorem (\uppercase\expandafter{\romannumeral2}) in \S 4]{Rohlin} that
    $(\pi^{-1}(\Omega_k),\mu)$ is isomorphic  Mod 0 to  $(\Omega_k\times \pi^{-1}(\omega) ,\hat{\mu}\times \mu_\omega)$ for any $\omega\in \Omega_k.$ 
	
\end{proof}

 Let  $(S,m)$ be a standard Borel probability space and $(Y,d)$ be a  separable metric space. Denote by $F(S,Y)$ the space of measurable maps $\phi:S\to Y$, where two maps are identified if they are equal almost everywhere. Then  $F(S,Y)$ is a separable metric space endowed with the metric
 \[d_F(\phi_1,\phi_2):=\inf\left\{\varepsilon>0:m(\{s\in S:d(\phi_1(s),\phi_2(s))>\varepsilon \})\leq \varepsilon\right\}. \]
 Moreover, if $Y$ is complete, then $F(S,Y)$ is also complete.
 
  The following three lemmas comes from \cite{Margulis1991Discrete} and \cite{Fisher04}, we shall give  a proof here for the completeness.
 \begin{lemma}[\cite{Margulis1991Discrete}]\label{lemma 3.8}
 	Let $\Omega,S$ be standard Borel spaces, $Y$ be a separable metric space, and $g:\Omega\times S\to Y$ be a  Borel map. Then
 	\begin{enumerate}
 		\item For any $\omega\in \Omega$, the map $g_\omega:S\to Y$ defined by $g_\omega(s)=g(\omega,s)$ is Borel.
 		\item  The induced map $\check{g}:\Omega\to F(S,Y)$ defined by $\check{g}(\omega)=g_\omega$ is Borel.
 	\end{enumerate}
 \end{lemma}
 \begin{proof}
 	(\romannumeral1 )  For any $\omega\in \Omega$ and any open subset $U\subset Y$, since g is Borel and 
 	  $$(g_\omega)^{-1}(U)=\left\{ s\in S:(\omega,s) \in g^{-1}(U) \right\}$$
 	is the $\omega$-section of $g^{-1}(U)$,  it follows that $g_\omega$ is Borel.
 	
 	(\romannumeral2 ) It's enough to prove for any $h\in F(S,Y)$ and any $\delta>0$,  the set $\{\omega\in \Omega:d_F({g}_\omega,h)<\delta  \}$ is measurable. Notice that 
 	\[\left\{\omega:d_F({g}_\omega,h)<\delta \right \}=\bigcup_{r_n\in \mathbb{Q}
 		\atop0<r_n<\delta}\left\{\omega:m(\{s\in S:d({g}_\omega(s),h(s))>r_n \})\leq r_n   \right\}. \]
 	Since  $\{s\in S:d({g}_\omega(s),h(s))>r_n \}$ is the $\omega$-section of the measurable set $\{(\omega,s)\in \Omega\times S:d({g}(\omega,s),h(s))>r_n \}$, by Fubini's theorem,  the measure of the $\omega$-section of the measurable set $\{(\omega,s)\in \Omega\times S:d({g}(\omega,s),h(s))>r_n \}$ is a  measurable function of $\omega$, that is, $\omega\mapsto m(\{s\in S:d({g}_\omega(s),h(s))>r_n \})$ is measurable, which implies the set 
 	  \[\left\{\omega:m(\{s\in S:d({g}_\omega(s),h(s))>r_n \})\leq r_n   \right\}\]
 	  is measurable. Thus we conclude the set $\left\{\omega:d_F({g}_\omega,h)<\delta \right \}$ is measurable. 	  
 \end{proof}

 The converse of this lemma is also true.
 \begin{lemma}[\cite{Fisher04}]\label{lemma 3.9}
 	Let $(\Omega,\hat{\mu}),(S,m)$ be standard Borel probability spaces, $Y$ be a separable complete metric space,  and $\Phi:\Omega\to F(S,Y)$ be a  Borel map. Then there exists a  Borel map $\hat{\Phi}:\Omega\times S\to Y$, such that for $\hat{\mu}$-a.e. $\omega\in \Omega$, 
 	\[\hat{\Phi}(\omega, s)=\Phi(\omega)(s),  \quad \text{for m-a.e.}~ s\in S. \]
 \end{lemma}
 \begin{proof}
 	Let $\left\{\mathcal{P}_n=\{P_{n,i}\}_{i\geq 1}\right\}_{n\geq 1}$ be a sequence of increasing partitions of $ F(S,Y)$    with the diameter of $P_{n,i}$ less that $2^{-n}$. Fix any $\phi_n^i\in P_{n,i}$. Define $\Phi_n:\Omega\times S\to Y$ by 
 	\[\Phi_n(\omega,s)=\phi_n^i(s), \text{~if~} \Phi(\omega)\in P_{n,i}. \]
 	Then for any $n,k\in \mathbb{N}, \omega\in \Omega$, one has
 	\begin{equation*}
 	m\{s:d(\Phi_n(\omega,s),\Phi_{n+k}(\omega,s))>2^{-n}   \}<2^{-n}.
 	\end{equation*}
 	Thus the Fubini theorem gives $(\hat{\mu}\times m)(\{(\omega,s):d(\Phi_n(\omega,s),\Phi_{n+k}(\omega,s))>2^{-n}\})\leq 2^{-n}$, which implies
 	$d_F(\Phi_n,\Phi_{n+k})\leq 2^{-n}.$ Thus $\{\Phi_n\}$ converges in $ F(\Omega\times S, Y)$. Denote by $\hat{\Phi}$ the limit of $\{\Phi_n\}$. Since $\{\Phi_n\}$ converges to $\hat{\Phi}$ in measure, there exists a subsequence $\{\Phi_{n_k}\}$ converges to $\hat{\Phi}$ almost everywhere. Let $\Phi_{n_k,\omega}:S\to X$ and $\hat{\Phi}_\omega:S\to Y$ be gotten by Lemma \ref{lemma 3.8}. Then for $\hat{\mu}$-a.e. $\omega\in \Omega$, $\Phi_{n_k,\omega}$ converges to $\hat{\Phi}_\omega $ almost everywhere, which implies $d_F(\Phi_{n_k,\omega},\hat{\Phi}_\omega)\to 0$.  Notice that for any $\omega,$  $d_F(\Phi_{n_k,\omega},\Phi(\omega))\leq 2^{-n_k}\to 0$ as $k\to\infty$. Therefore  $\hat{\Phi}_\omega=\Phi(\omega) $ for $\hat{\mu}$-a.e. $\omega\in \Omega$. This proves the Borel map $\hat{\Phi}\in F(\Omega\times S, Y)$ satisfies 
 	\[\hat{\Phi}(\omega,s)=\hat{\Phi}_\omega(s)=\Phi(\omega)(s), ~~\text{for }~\hat{\mu}\text{-a.e.~} \omega\in \Omega \text{~and m-a.e.~} s\in S.\]
 \end{proof}

 Let $(S,m)$ be a standard Borel probability space, and denote by $Aut_{m}(S)$ the group of measure preserving automorphisms on S, where two automorphisms are identified if they are equal almost everywhere. Since $S$ is a standard Borel space, we may assume $S$ is a separable  and complete metric space, and hence $Aut_{m}(S)$ is a closed subset of $F(S,S)$.
 \begin{lemma}[\cite{Fisher04}]\label{lemma 3.10}
 	Let $(S,m)$ be a standard Borel probability space, $Y$ be a  separable metric space. Then the natural action  $Aut_{m}(S)\times F(S,Y)\to F(S,Y)$ defined by $(g,\phi)\mapsto \phi\circ g$ is continuous.
 \end{lemma}
 \begin{proof}
 	Given any $g_n\to g, \phi_n\to \phi$, we'd prove $d_F(\phi_n\circ g_n,\phi\circ g)\to 0$.
 	
 	Note that $d_F(\phi_n\circ g_n,\phi\circ g)\leq d_F(\phi_n\circ g_n,\phi\circ g_n)+d_F(\phi\circ g_n,\phi\circ g)$. 
 	since 
 	\begin{align*}
 	m\{s:d\big(\phi_n\circ g_n(s),\phi\circ g_n(s)\big)>\varepsilon\} & = m(g_n^{-1}\{s:d(\phi_n,\phi)>\varepsilon\})\\
 	& = m\{s:d(\phi_n,\phi)>\varepsilon\}\to 0,
 	\end{align*}
 	we have $d_F(\phi_n\circ g_n,\phi\circ g_n)\to 0.$ Then it remains to show $d_F(\phi\circ g_n,\phi\circ g)\to 0$.
 	
 	For any $\varepsilon>0$, by  Lusin's theorem, there exists a compact subset $K_\varepsilon\subset S$ such that  $\phi$ is uniformly continuous on $K_\varepsilon$ and $m(K_\varepsilon)>1-\frac{1}{4}\varepsilon.$ Then there exists $\delta>0$ such that for any
 	$x,y\in K_\varepsilon$ with $d(x,y)\leq\delta$, one has $d(\phi(x),\phi(y))\leq \varepsilon$. Let $A_n=g_n^{-1}(S\setminus K_\varepsilon)\cup g^{-1}(S\setminus K_\varepsilon)$. Since $m$ is $g_n$ and $g$ invariant, one sees $m(A_n)\leq \frac{\varepsilon}{2}.$ Then equation 
 	\[\left\{s:d\big(\phi\circ g_n(s),\phi\circ g(s)\big)>\varepsilon \right\}\subset A_n\cup\{s:d\big(g_n(s),g(s)\big)>\delta\}  \]
 	gives 
 	\begin{align*}
 	\limsup_{n\to\infty} m\{s:d\big(\phi\circ g_n(s),\phi\circ g(s)\big)>\varepsilon\}\leq \frac{\varepsilon}{2}+0<\varepsilon.
 	\end{align*} Since $\varepsilon$ is arbitrary, we conclude that 
 	$d_F(\phi\circ g_n,\phi\circ g)\to 0$. 
 \end{proof}

{\bf Reductions}. By Theorem \ref{thm 3.5} and Proposition  \ref{prop 3.7},  $M$ is isomorphic Mod 0 to the disjoint union of at most countably many  spaces $\Omega_k\times S_k$. Hence we may assume without loss of generality that $M=\Omega\times S$ and $\mu=\hat{\mu}\times m$, where $(\Omega,\hat{\mu}),(S,m)$ are standard Borel probability spaces. By the ergodic decomposition theorem,  every $\{\omega\}\times S$  is $f$-invariant.  Denote $f(\omega,s)=(\omega,f_2(\omega,s))$. Then by  Theorem \ref{thm 3.5} and  Lemma \ref{lemma 3.8}, for $\hat{\mu}$-a.e. $\omega\in \Omega$, the map $f_\omega:=f_{2,\omega}:S\to S $ defined by $f_\omega(s)=f_2(\omega,s)$ is Borel and preserves the ergodic measure $m.$ Therefore, the condition of Theorem \ref{thm 3.4} becomes: for $\hat{\mu}$-a.e. $\omega\in \Omega,$ there exists a measurable map $C_\omega:S\to K,$ such that
\begin{equation}\label{3.9}
A(\omega,s)=C_\omega(f_\omega(s))C_\omega(s)^{-1}, \text{~for~} m\text{-}a.e. ~s\in S.
\end{equation}  	 

%
%
%
As a corollary of Lemma \ref{lemma 3.10}, we have the following result.
 
 \begin{corollary}\label{cor 3.11}
 	Let $\Omega,S,f,$ be as above, and $K$ be as in  \eqref{3.8}, Then the map defined by  $\Omega\times F(S,K)\to F(S,K)$, $(\omega,\phi)\mapsto \phi(f_\omega)$ is Borel. 
 \end{corollary}
 \begin{proof}
 	By Lemma \ref{lemma 3.8}, the map $\omega\mapsto f_\omega$ is Borel.  Thus we obtain the map
 	\begin{align*}
 	\Omega\times F(S,K)& \to Aut_{m}(S)\times F(S,K)\\
 	(\omega,\phi) & \mapsto (f_\omega,\phi)
 	\end{align*}
 	is Borel. By Lemma \ref{lemma 3.10}, the natural action $Aut_{m}(S)\times F(S,K)\to F(S,K)$, $(g,\phi)\mapsto \phi\circ g$ is continuous .   We conclude $(\omega,\phi)\mapsto \phi(f_\omega)$ is Borel. 	
 \end{proof}
 Denote 
 \[K^\prime=  \overline{\bigcup\limits_{m,n\in \mathbb{Z}}\left\{\mathcal{A}_y^m\mathcal{A}_x^n:x,y\in M\right\}}\subset GL(X).  \]
 Then $K^\prime$ is a  separable, complete metric space. We have the following lemma.
 \begin{lemma}\label{lemma 3.12}
 	  The  map   $\tau_1:F(S,K)\to F(S,K)$ defined by 
 	  $$\tau_1(\phi)(s)=\phi(s)^{-1}$$
 	and the map  
 	  $\tau_2:F(S,K)\times F(S,K)\to F(S,K^\prime)$ defined by 
 	$$\tau_2(\phi,\psi)(s)=\phi(s)\circ\psi(s)$$
  are  continuous. 
 \end{lemma}
 
 \begin{proof}
 	Suppose that $\phi_n\to \phi\in F(S,K)$, that is,
 	\[\inf\left\{\varepsilon>0:m(\{s\in S:d(\phi_n(s),\phi(s))>\varepsilon \})\leq \varepsilon\right\}\to 0, \mbox{as~} n\to \infty. \]
 	Since 
 	\[ d(\phi_n(s),\phi(s))=\|\phi_n(s)-\phi(s)\|+\|\phi_n(s)^{-1}-\phi(s)^{-1}\|=d(\phi_n(s)^{-1},\phi(s)^{-1}),\]
 	 we have $\phi_n^{-1}\to \phi^{-1}\in F(S,K)$ as $n\to\infty$. That is, $\tau_1$ is continuous.

 	Now given any sequences $\phi_n\to \phi\in F(S,K),\psi_n\to \psi\in F(S,K)$, to show $\phi_n\psi_n\to \phi\psi\in F(S,K^\prime),$  it suffices to prove for any $\varepsilon>0$, 
 	\[\limsup_{n\to\infty} m\left( \{  s:d(\phi_n(s)\psi_n(s),\phi(s)\psi(s))>\varepsilon  \}  \right)\leq \varepsilon. \]
 	Indeed, if this equation holds, then there exists $N\geq 1$, such that for any $n\geq N,$ one has $$m\left( \{  s:d(\phi_n(s)\psi_n(s),\phi(s)\psi(s))>2\varepsilon  \}  \right)\leq m\left( \{  s:d(\phi_n(s)\psi_n(s),\phi(s)\psi(s))>\varepsilon  \}  \right)\leq 2\varepsilon.$$
  	Since $\varepsilon$ is arbitrary, it implies $\phi_n\psi_n\to \phi\psi$ as $n\to\infty.$
  	
  	Notice that
 	\begin{align*}
                   	&\left\{  s:d(\phi_n(s)\psi_n(s),\phi(s)\psi(s))>\varepsilon  \right\}\\
\subset  	&\left\{  s:\|\phi_n(s)\psi_n(s)-\phi(s)\psi(s)\|>\frac{\varepsilon }{2} \right\} \bigcup \left\{  s:\|\psi_n(s)^{-1}\phi_n(s)^{-1}-\psi(s)^{-1}\phi(s)^{-1}\|>\frac{\varepsilon}{2}  \right\} . 
 	\end{align*}
 	since $\|\phi_n(s)\psi_n(s)-\phi(s)\psi(s)\|\leq \|\phi_n(s)-\phi(s)\|\cdot\|\psi_n(s)\|+\|\phi_(s)\|\cdot\|\psi_n(s)-\psi(s)\|$, we have
 	\begin{align*}
 	&\left\{  s:\|\phi_n(s)\psi_n(s)-\phi(s)\psi(s)\|>\frac{\varepsilon }{2} \right\}\\
 	\subset &\left\{  s:\|\phi_n(s)-\phi(s)\|\cdot\|\psi_n(s)\|>\frac{\varepsilon }{4} \right\}\bigcup\left\{  s:\|\phi_(s)\|\cdot\|\psi_n(s)-\psi(s)\|>\frac{\varepsilon }{4}\right\}.
 	\end{align*}
   We shall prove:  $\limsup\limits_{n\to\infty} m(\{  s:\|\phi_n(s)-\phi(s)\|\cdot\|\psi_n(s)\|>\frac{\varepsilon }{4} \})\leq \frac{\varepsilon}{4}.$
 	Let $A_n=\{s: \|\psi_n(s)-\psi(s)\|\leq1 \}$. Then we claim that $\lim\limits_{n\to\infty}m(S\setminus A_n)=0.$ Indeed, since $\psi_n\to \psi$ in  measure,  for any $0<\delta<1, $ we have 
 	\[\limsup_{n\to\infty}m(S\setminus A_n)\leq \limsup_{n\to\infty}m(\{ s: \|\psi_n(s)-\psi(s)\|>\delta \})\leq \delta. \]
 	Since $\delta$ is arbitrary, we conclude $\lim_{n\to\infty}m(S\setminus A_n)=0.$ Denote $D_n:=\{s:\|\psi(s)\|\leq  n \}$.  Since  $\lim\limits_{n\to\infty}m(D_n)=\lim\limits_{n\to\infty}m(\cup_{n\geq 1}D_n)=1$, we may take $N$ large enough such that  $m(D_N)>1-\frac{\varepsilon}{4}.$
 	Then we have
 	\begin{align*}
 	&\left\{  s:\|\phi_n(s)-\phi(s)\|\cdot\|\psi_n(s)\|>\frac{\varepsilon }{4} \right\}\\
 	\subset& (S\setminus A_n)\bigcup \left\{  s:\|\phi_n(s)-\phi(s)\|>\frac{\varepsilon /4}{\|\psi(s)\|+1} \right\}\\
 	\subset & (S\setminus A_n)\bigcup(S\setminus D_N)\bigcup \left\{  s:\|\phi_n(s)-\phi(s)\|>\frac{\varepsilon /4}{N+1} \right\}.
 	\end{align*}
 	Since $\phi_n\to \phi$,  we have $\limsup\limits_{n\to\infty}m(\left\{  s:\|\phi_n(s)-\phi(s)\|\cdot\|\psi_n(s)\|>\frac{\varepsilon }{4} \right\})\leq \frac{\varepsilon}{4}.$ Similarly, we can also get $\limsup\limits_{n\to\infty}m( \left\{  s:\|\phi_(s)\|\cdot\|\psi_n(s)-\psi(s)\|>\frac{\varepsilon }{4}\right\})\leq \frac{\varepsilon}{4}.$ Therefore,  
 	\[\limsup_{n\to\infty}m(\left\{  s:\|\phi_n(s)\psi_n(s)-\phi(s)\psi(s)\|>\frac{\varepsilon }{2} \right\} )\leq \frac{\varepsilon}{2}.\]
 	It can be proved analogously that
   $$\limsup_{n\to\infty} m(	\left\{  s:\|\psi_n(s)^{-1}\phi_n(s)^{-1}-\psi(s)^{-1}\phi(s)^{-1}\|>\frac{\varepsilon}{2}  \right\})\leq \frac{\varepsilon}{2}.$$
 	 Thus $\limsup\limits_{n\to\infty}m(\left\{  s:d(\phi_n(s)\psi_n(s),\phi(s)\psi(s))>\varepsilon  \right\})\leq \varepsilon.$
 	     This proves  $\phi_n\psi_n\to \phi\psi,$ that is, $\tau_2$ is continuous.	
 \end{proof}
 
 We also need the following  von Neumann's selection theorem to find the desired $\mu$-continuous map $C:M\to GL(X).$ Recall that a subset $\mathcal{F}$ of a standard Borel space $\Sigma$ is called {\em analytic} if there is a  standard Borel space $Y$ and a Borel map $\psi: Y\to \Sigma$ such that $\mathcal{F}=\psi(Y).$
  \begin{theorem}[von Neumann Selection Theorem \text{\cite[Theorem 3.4.3]{Arveson}}]\label{thm 3.13}
 	Let $(\Omega,\hat{\mu})$ be a standard Borel probability space, $L$ be a standard Borel space, and $\mathcal{F}$ be an analytic subset of $\Omega\times L$. Denote by $\Omega_\mathcal{F}$ the projection of  $\mathcal{F}$ to $\Omega$. Then there exists a Borel subset $\Omega_0\subset \Omega_\mathcal{F}$ of $\hat{\mu}$-full measure and a  Borel function $\Phi:\Omega_0\to L$, such that {\em graph}$(\Phi)\subset \mathcal{F}.$
 \end{theorem}

 We now continue the proof of Theorem \ref{thm 3.4}.
 
By the Reductions in this subsection, we  assume $M=\Omega\times S$ and $\mu=\hat{\mu}\times m$, where $(\Omega,\hat{\mu}),(S,m)$ are standard Borel probability spaces.  Denote $f(\omega,s)=(\omega,f_\omega(s))$. Then  for any $\omega\in \Omega,$ $f_\omega(s)\in Aut_m(S)$. 
   By Lemma \ref{lemma 3.8}, the map $A:\Omega\times S\to K$ induces   the Borel maps $A_\omega:S\to K$  and $\omega\mapsto A_\omega$.
    Denote 
   \[\widetilde{K}=  \overline{\bigcup\limits_{k,m,n\in \mathbb{Z}}\left\{\mathcal{A}_x^k\mathcal{A}_y^m\mathcal{A}_z^n:x,y,z\in M\right\}}\subset GL(X).  \]
   Then $\widetilde{K}$ is a  separable, complete metric space. We define 
 	\[\sigma:\Omega\times F(S,K)\to F(S,\widetilde{K})  \] 
 	by 
 	$$\sigma(\omega,\phi)(s)=\phi(f_\omega(s))^{-1}A_\omega(s)\phi(s).$$
 	Then $\sigma$ is Borel. Indeed, Corollary \ref{cor 3.11} and Lemma \ref{lemma 3.12} show that the map   $(\omega,\phi)\mapsto (\phi\circ f_\omega)^{-1}$ is Borel. Since the second term $(\omega,\phi)\mapsto A_\omega$ and the third term  $(\omega,\phi)\mapsto \phi$ are also Borel, using the second conclusion of Lemma \ref{lemma 3.12}, we conclude  $\sigma$ is Borel measurable.  
 	
 	Denote $\mathcal{F}=\sigma^{-1}(\{Id\})$. Then $\mathcal{F}$ is an analytic subset of $\Omega\times F(S,K)$. Denote by $\Omega_\mathcal{F}$ the projection of  $\mathcal{F}$ onto $\Omega$. Then by \eqref{3.9}, $\Omega_\mathcal{F}$ is of $\hat{\mu}$-full measure. Then by  Von Neumann selection theorem \ref{thm 3.13}, there exists a  Borel map $\Phi:\Omega\to F(S,K)$, such that for $\hat{\mu}$-a.e. $\omega\in \Omega$, one has $(\omega,\Phi(\omega))\in \mathcal{F}$,  that is,
 	for $\hat{\mu}$-a.e. $\omega\in \Omega$, 
 	\[A(\omega,s)=\Phi(\omega)(f_\omega(s))\circ\Phi(\omega)(s)^{-1},\text{~for~} m\text{-}a.e.~ s\in S.\]
 	Hence by Lemma \ref{lemma 3.9}, there exists a  Borel map $C:\Omega\times S\to K$, such that for $\hat{\mu}$-a.e. $\omega$, we have $$C(\omega,s)=\Phi(\omega)(s), \text{~for~} m\text{-}a.e.~ s\in S.$$
 	Therefore, $A(x)=C(fx)C(x)^{-1}$, for $\mu=(\hat{\mu}\times m)$-a.e. $x=(\omega,s)\in M$. Since $K$ is separable, by Lusin's theorem, C is $\mu$-continuous. This completes the proof of Theorem \ref{thm 3.4}.
 \end{proof}
 \section{Proof of Theorem \ref{thm B}}
 We begin by reducing the proof to  the topologically mixing case.  Since $\mu$ is  an  ergodic $f$-invariant measure on M with full support, one has $f$ is transitive. By the spectral decomposition theorem, there is an integer $k\geq 1$ such that $M=\bigsqcup_{i=1}^k \Sigma_i$, satisfying $f(\Sigma_i)=\Sigma_{i+1}$ and  $f^k|_{\Sigma_i}$ is topologically mixing for $1\leq i\leq k$. Then the normalized restriction $\mu_{\Sigma_i}$ of $\mu$ to $\Sigma_i$ is an ergodic  $f^k$-invariant measure on $\Sigma_i$ with full support and local product structure, and $\mathcal{A}^k$ is a  H\"{o}lder continuous cocycle for $f^k$ satisfying $\mathcal{A}_x^k=C(f^kx)C(x)^{-1}$ for $\mu_{\Sigma_i}$-a.e. $x\in \Sigma_i.$ Assuming Theorem \ref{thm B} holds for topologically mixing systems $f^k|_{\Sigma_i}$, then C coincides $\mu_{\Sigma_i}$-a.e. with a H\"{o}lder continuous map. Since  $M=\bigsqcup_{i=1}^k \Sigma_i$ is a disjoint union, we conclude that C coincides $\mu$-a.e. with a H\"{o}lder continuous map satisfying the equation \eqref{1.2}.
 
 Now we begin the proof of Theorem \ref{thm B} in the topologically mixing case.
 
 \subsection{Extending C to $D(N,\theta)$.}
 Since there exists a $\mu$-continuous map $C:M\to GL(X)$  satisfying \eqref{1.2},
 we can find a compact set $\hat{K}$ with positive $\mu$-measure on which $C$ is continuous and thus the norms of $C$ and $C^{-1}$ are bounded. Then by the Poincar\'{e}'s recurrence theorem, for $\mu$-a.e. $x\in \hat{K}$, there exists infinitely many $n_k$ with $f^{n_k}(x)\in\hat{K}.$ We may take such a point $x\in \hat{K}$ whose iterations satisfy \eqref{1.2}, and we may also assume the point $x$ is regular, that is,  $x$ satisfies $\lambda_+(\mathcal{A},\mu)=\lim\limits_{n\to\infty}\frac{1}{n}\log\|\mathcal{A}_x^n \|$.  Therefore, 
 \begin{align*}
 \lambda_+(\mathcal{A},\mu) =\lim\limits_{k\to\infty}\frac{1}{n_k}\log\|\mathcal{A}_x^{n_k} \|
 =\lim\limits_{k\to\infty}\frac{1}{n_k}\log\|C(f^{n_k}x)C(x)^{-1} \|
 =0.
 \end{align*}
 Similarly,   $\lambda_-(\mathcal{A},\mu)=\lim\limits_{n\to\infty}-\frac{1}{n}\log\|(\mathcal{A}_x^{n})^{-1} \|=0.$  
 Then by \cite[Lemma 2.2]{Viana08}, for any $\theta>0$ and $\mu$-a.e. $x\in M$, there exists $N\geq 1$ such that
 \[\prod_{j=0}^{k-1}\|\mathcal{A}_{f^{jN}x}^{N}\|\leq e^{kN\theta}, ~\text{and}~ \prod_{j=0}^{k-1}\|(\mathcal{A}_{f^{jN}x}^{N})^{-1}\|\leq e^{kN\theta}, \quad \forall k\geq 1. \]
 Thus 
 \[ \prod_{j=0}^{k-1}\|\mathcal{A}_{f^{jN}x}^{N}\|\cdot\|(\mathcal{A}_{f^{jN}x}^{N})^{-1}\|\leq e^{kN\theta}, \quad \forall k\geq 1.\]
 Analogously, since $\lambda_+(\mathcal{A},\mu)=\lim\limits_{n\to \infty}\frac{1}{n}\log\|\mathcal{A}_{f^{-n}x}^{n} \|$ and $\lambda_-(\mathcal{A},\mu)=\lim\limits_{n\to \infty}-\frac{1}{n}\log\|\mathcal{A}_{x}^{-n} \|$  \cite[Section 2]{Kalinin16},  equation \eqref{2.2} also holds. Hence for any $\theta>0$ and $\mu$-a.e. $x\in M$, there exists $N\geq 1$ such that $x\in D(N,\theta).$ Take $\theta<\tau\alpha.$ Then by Proposition \ref{prop 2.4},  the stable and unstable holonomies exist almost everywhere.
 
 \begin{lemma}\label{lemma 4.1}
      There exists a set  $\Omega$ of full $\mu$-measure   such that for any $x,y\in \Omega$ with $y\in W^*_{loc}(x)$, one has 
      \[H^*_{x,y}C(x)=C(y),  \] 
     where $*\in\{s,u \}$.    
 \end{lemma} 
 \begin{proof}
 	Fix a $\theta<\tau\alpha.$ Then for any given $N\geq 1$, and for $\mu$-a.e. $x,y\in D(N,\theta)$ with $y\in W^s_{loc}(x)$, one has 
 	\begin{equation}\label{4.1}
 	H^s_{x,y}=\lim\limits_{n\to\infty}(\mathcal{A}_y^n)^{-1}\mathcal{A}_x^n=\lim\limits_{n\to\infty}C(y)C(f^ny)^{-1}C(f^nx)C(x)^{-1}. 
 	\end{equation}
 	Since $C$ is $\mu$-continuous,  we may take a compact subset K with $\mu(K)>\frac{1}{2}$ such that C is continuous on K. By  	Birkhoff  Ergodic Theorem, for $\mu$-a.e. $z\in M$, $\frac{1}{n}\Sigma_{i=0}^{n-1}\chi_K(f^iz)\to\mu(K)>\frac{1}{2}$ as $n\to\infty.$ Thus for  $\mu$-a.e. $x,y\in D(N,\theta)$ with $y\in W^s_{loc}(x)$,, there exists a sub-sequence $\{n_i\}_{i\geq 1}$  such that $f^{n_i}(x),f^{n_i}(y)\in K$ for any $i\geq 1$. Then \eqref{4.1} gives
 	\[	H^s_{x,y}=\lim\limits_{i\to\infty}C(y)C(f^{n_i}y)^{-1}C(f^{n_i}x)C(x)^{-1}=C(y)C(x)^{-1}. \]
 	Analogously,  we can also get that for $\mu$-a.e. $x,y\in D(N,\theta)$ with $y\in W^u_{loc}(x)$,  $	H^u_{x,y}=C(y)C(x)^{-1}. $
 	Taking the union over all $N\geq 1$ of the sets $D(N,\theta)$,  we obtain a   full $\mu$-measure set 
 	$\Omega$,    such that for any $x,y\in \Omega$ with $y\in W^*_{loc}(x)$, one has 
 $H^*_{x,y}C(x)=C(y), $ 	where $*\in\{s,u \}.$   	
 \end{proof}
 
 The following lemma shows that the local product structure of $\mu$ and the  holomony invariance of C imply the map C can be extended continuously to $\text{supp}(\mu|{D(N,\theta)})$.
 
 \begin{lemma}\label{Lemma 4.2}
   For any $\theta<\tau\alpha$ and $N\geq 1,$ there exists an $\alpha$-H\"older continuous map $\widehat{C}$ defined on $\text{supp}(\mu|{D(N,\theta)})$ which coincides $\mu$-a.e.  on   $\text{supp}(\mu|{D(N,\theta)})$ with C.
 \end{lemma}
 \begin{proof}
 	Let $\delta>0$ be small enough such that for any $y,z\in M$ with $d(y,z)<2\delta$, $W^s_{loc}(y)$ intersects $W^u_{loc}(z)$ at exactly one point $[y,z].$ Given any $x\in D(N,\theta)$, denote
 	\[\mathcal{N}_x^u(\delta)=\mathcal{N}_x^u(N,\theta,\delta):=\{[y,x]:y\in B(x,\delta)\cap D(N,\theta)\},  \]
 	\[\mathcal{N}_x^s(\delta)=\mathcal{N}_x^s(N,\theta,\delta):=\{[x,y]:y\in B(x,\delta)\cap D(N,\theta)\}.  \]
 	Then $\mathcal{N}_x^*(\delta)\subset W^*_{loc}(x)$ for $*\in \{s,u\}$. Let $\mathcal{N}_x(\delta)=[\mathcal{N}_x^u(\delta),\mathcal{N}_x^s(\delta)]$ be the image of $\mathcal{N}_x^u(\delta)\times \mathcal{N}_x^s(\delta)$ under the map  $(y,z)\mapsto [y,z]$. 
    Since $\mu$ has local product structure, one has
    \[\text{supp}(\mu|\mathcal{N}_x(\delta))=[\text{supp}(\mu^u|\mathcal{N}_x^u(\delta)),\text{supp}(\mu^s|\mathcal{N}_x^s(\delta))], \]
    where $\mu^u|\mathcal{N}_x^u(\delta)$ and $\mu^s|\mathcal{N}_x^s(\delta)$ are the projections of $\mu|\mathcal{N}_x(\delta)$ to $\mathcal{N}_x^u(\delta)$ and $\mathcal{N}_x^s(\delta)$	respectively.
   Notice that $\mathcal{N}_x(\delta)\supset D(N,\theta)\cap B(x,\delta)$. It suffices to construct an $\alpha$-H\"older continuous map $\widehat{C}$ on $\text{supp}(\mu| {\mathcal{N}_x(\delta)})$ which coincides $\mu$-a.e.  on $\text{supp}(\mu| {\mathcal{N}_x(\delta)})$  with C.	
   
   Since $\mu$ has local product structure and $\mu(\Omega)=1$,  it gives that for $\mu^u$-a.e. $\xi\in \mathcal{N}^u_x(\delta),$ 
   \[\mu^s([\xi,\mathcal{N}^s_x(\delta)]\setminus \Omega)=0.\]
   Fix any such $\xi.$ Let $\Omega_\mathcal{N}$ be the set of points in $\mathcal{N}_x(\delta)\cap \Omega$ that lie on the local unstable leaves of $[\xi,\mathcal{N}^s_x(\delta)]\cap \Omega.$ Then we have
     \[\mu\left(\mathcal{N}_x(\delta)\setminus \Omega_\mathcal{N}\right) =0. \]
     Fix  $x_0\in [\xi,\mathcal{N}^s_x(\delta)]\cap \Omega$. For any $z\in \Omega_\mathcal{N},$ let $\eta=[x_0,z].$ Then by the construction of $\Omega_\mathcal{N}$, one has $\eta\in \Omega_\mathcal{N}.$ By Lemma \ref{lemma 4.1},  $C(z)=H^u_{\eta,z}H^s_{x_0,\eta}C(x_0).$ Define $\widehat{C}$ on $\text{supp}(\mu| {\mathcal{N}_x(\delta)})$ by 
     \[\widehat{C}(z):=H^u_{[x_0,z],z}H^s_{x_0,[x_0,z]}C(x_0), \quad \forall z\in  \text{supp}(\mu| {\mathcal{N}_x(\delta)}).\] Then by the construction, $\widehat{C}=C$ almost everywhere  on $\text{supp}(\mu| {\mathcal{N}_x(\delta)})$. Since the stable and unstable holonomies are  continuous, we conclude that $\widehat{C}$ is  continuous on $\Omega_\mathcal{N}$. Moreover,     by construction,  
     $$\widehat{C}(z)=H^u_{y,z}\widehat{C}(y)  \quad \forall z\in  \text{supp}(\mu| {\mathcal{N}_x(\delta)}),y\in [\mathcal{N}^u_x(\delta),z].$$
     That is, $\widehat{C}$ is invariant under unstable holonomies on $\text{supp}(\mu| {\mathcal{N}_x(\delta)}).$
     
     By a dual procedure, we can obtain a continuous map $\widetilde{C}$ on $\text{supp}(\mu| {\mathcal{N}_x(\delta)})$ which is invariant under stable holonomies and  coincides $\mu$-a.e.   with C  on $\text{supp}(\mu| {\mathcal{N}_x(\delta)}).$ Then by the continuity, $\widehat{C}=\widetilde{C}$. Hence $\widehat{C}$ is invariant under both stable and unstable holonomies on $\text{supp}(\mu| {\mathcal{N}_x(\delta)}).$  By Proposition \ref{prop 2.4},  $\|H_{y,z}^*-Id\|\leq L\cdot d(y,z)^\alpha$ for $y,z\in \mathcal{N}_x(\delta)$ and $y\in W^*_{loc}(z)$.  It follows that $\widehat{C}$ is  $\alpha$-H\"older continuous on every stable and unstable leaf, and thus  $\alpha$-H\"older continuous  on $\text{supp}(\mu| {\mathcal{N}_x(\delta)}).$
 \end{proof}
 
 Since $\text{supp}(\mu|{D(N,\theta)})$ may be a proper subset of $D(N,\theta)$, we use the following lemma to  obtain  a  continuous map $\widehat{C}$ defined on $D(N,\theta)$. This lemma resembles Lemma 4.5 of \cite{Butlerconformal}. We give a geometric proof here.
 \begin{lemma}\label{Lemma 4.3}
 	For any $\theta<\tau\alpha$ and $N\geq 1,$ there exist $\theta<\theta_*<\tau\alpha$ and $N_*\geq N,$ such that  
 	$$D(N,\theta)\subset\text{supp}(\mu|{D(N_*,\theta_*)}).$$
 \end{lemma}
 \begin{proof}
    For any  $\theta<\tau\alpha$, since $\lambda_+(\mathcal{A},\mu)=\lambda_-(\mathcal{A},\mu)=0$, we may take $N\geq 1$ large enough such that \[\frac{1}{N}\int\log(\|\mathcal{A}_x^N\|\cdot\|(\mathcal{A}_x^N)^{-1}\|)d\mu<\theta.\] 
    Set $\varphi(x)=\frac{1}{N}\log(\|\mathcal{A}_x^N\|\cdot\|(\mathcal{A}_x^N)^{-1}\|).$ 
    Given any $\gamma<(\tau\alpha-\theta)/3$,  denote
    \[K_{J,\gamma}=\{y\in M:\frac{1}{n}\sum_{i=0}^{n-1}\varphi(f^iy)\leq \int\varphi d\mu+\gamma<\theta+\gamma, ~~\forall n\geq J\}.\]
    Then $K_{1,\gamma}\subset K_{2,\gamma} \subset\cdots$, and the Birkhoff Ergodic Theorem gives $\mu(\bigcup_{j\geq 1} K_{j,\gamma})=1$. Thus $\lim\limits_{j\to \infty}\mu(K_{j,\gamma})=1.$ Take $\delta>0$ small enough such that for any $y,z\in M$ with $d(y,z)<3\delta$,  one has $|\varphi(y)-\varphi(z)|<\frac{\gamma}{2}$.  Then $W^s_{3\delta}(K_{J,\gamma/2})\subset K_{J,\gamma}.$
     Consider $U_x=[W^u_\delta(x),W^s_\delta(x)]$, that is, the image of $W^u_\delta(x)\times W^s_\delta(x)$ under the map  $(\xi,\eta)\mapsto [\xi,\eta]$. Then we have 
     \begin{equation}\label{4.2}
     [y,W^s_\delta(z)]\subset K_{J,\gamma},\quad \forall y\in K_{J,\gamma/2}\cap U_z.
     \end{equation}
      Since $\mu$ has full support, we may fix $J\geq1 $ large enough such that for any $z\in M,$ 
       \begin{equation}\label{4.3}
     \mu(K_{J,\gamma/2}\cap U_z)>0.
      \end{equation}
      
    Now for any $x\in D(N,\theta)$, and any $y\in B_m(x,2\delta)\cap f^{-m}(K_{J,\gamma})$ for some $m\geq 1$, where 
    $B_m(x,2\delta)=\{y: d(f^ix,f^iy)<2\delta,~\forall 0\leq i\leq m-1\},$  we estimate $\frac{1}{n}\sum_{i=0}^{n-1}\varphi(f^iy).$    If $n\leq m,$ then
    \begin{align*}
    \frac{1}{n}\sum_{i=0}^{n-1}\varphi(f^iy)&=\frac{1}{n}\sum_{i=0}^{n-1}\big(\varphi(f^iy)-\varphi(f^ix) \big)+\frac{1}{n}\sum_{i=0}^{n-1}\varphi(f^ix)\\
                                                                     &\leq \frac{\gamma}{2}+\theta.
    \end{align*}
    If $n>m,$ then 
  \begin{equation*}
     \begin{split}
    \frac{1}{n}\sum_{i=0}^{n-1}\varphi(f^iy)&=\frac{1}{n}\sum_{i=0}^{m-1}\big(\varphi(f^iy)-\varphi(f^ix) \big)+\frac{1}{n}\sum_{i=0}^{m-1}\varphi(f^ix)+\frac{1}{n}\sum_{i=m}^{n-1}\varphi(f^iy)\\
    &\leq \frac{m}{n}\frac{\gamma}{2}+\frac{m}{n}\theta+\frac{1}{n}\sum_{j=0}^{n-m-1}\varphi(f^j(f^my)).
    \end{split}
  \end{equation*}
    Since $f^m(y)\in K_{J,\gamma},$ we have 
    \begin{align*}
    \frac{1}{n}\sum_{j=0}^{n-m-1}\varphi(f^j(f^my))\leq  
    \begin{cases}
    	\frac{n-m}{n}(\theta+\gamma), & \mbox{if $n-m\geq J$},\\
    		\frac{J}{n}\|\varphi\|, &\mbox{if $n-m<J$}.
    \end{cases}
\end{align*}
  Take R large enough such that $	\frac{J}{R}\|\varphi\|\leq \gamma.$ Then for any $n\geq R,$  one has 
  \[\frac{1}{n}\sum_{j=0}^{n-m-1}\varphi(f^j(f^my))\leq  \frac{n-m}{n}(\theta+\gamma)+\gamma\leq \frac{n-m}{n}\theta+2\gamma.   \]
  Thus we conclude that for any $x\in D(N,\theta)$,  $y\in B_m(x,2\delta)\cap f^{-m}(K_{J,\gamma})$, and any $n\geq R$, we have
  \[ \frac{1}{n}\sum_{i=0}^{n-1}\varphi(f^iy)\leq \theta+3\gamma. \]
 Let $N_1=NR$, then for  any $x\in D(N,\theta)$,  $y\in B_m(x,2\delta)\cap f^{-m}(K_{J,\gamma})$, we have
   \[ \prod_{j=0}^{k-1}\|\mathcal{A}_{f^{jN_1}y}^{N_1}\|\cdot\|(\mathcal{A}_{f^{jN_1}y}^{N_1})^{-1}\|\leq e^{N\sum_{j=0}^{kR-1}\varphi(f^jy)}\leq  e^{kN_1(\theta+3\gamma)}, \quad \forall k\geq 1.\]
   We claim that  for any $x\in D(N,\theta)$,  $ B_m(x,2\delta)\cap f^{-m}(K_{J,\gamma})$ contains a $\mu$-positive subset which is $W^s$-saturated on $U_x$. Indeed, let 
   \[K^u :=f^{-m}([K_{J,\gamma/2}\cap U_{f^mx},f^m(x)])\subset W^u_\delta(x).\]
   Then
   $K_{J,\gamma/2}\cap U_{f^mx}\subset [f^m(K^u),W^s_\delta(f^mx)]. $
   Since $[f^m(K^u),W^s_\delta(f^mx)]\subset[K_{J,\frac{\gamma}{2}}\cap U_{f^mx},W^s_\delta(f^mx)]$,
    by \eqref{4.2},  we obtain
    $$K_{J,\gamma/2}\cap U_{f^mx}\subset[f^m(K^u),W^s_\delta(f^mx)]\subset K_{J,\gamma}.$$
   By \eqref{4.3} and $\mu$ is $f$-invariant,  $$\mu\left(f^{-m}( [f^m(K^u),W^s_\delta(f^mx)])\right)=\mu\left( [f^m(K^u),W^s_\delta(f^mx)]\right)\geq \mu(K_{J,\gamma/2}\cap U_{f^mx})>0.$$
   Since   $f^{-m}( [f^m(K^u),W^s_\delta(f^mx)])$ is $W^s$-saturated on $f^{-m}(U_{f^mx})$ and $\mu$ has local product structure, it implies that 
   \[\mu^u(K^u)>0, \quad\mbox{and} ~~[K^u,W^s_\delta(x)]\subset f^{-m}(K_{J,\gamma}). \] 
   Since $K^u\subset f^{-m}(W^u_\delta(f^mx))\subset B_m(x,\delta)$, one has 
$[K^u,W^s_\delta(x)]\subset B_m(x,2\delta)$.
Thus we conclude that 
$$[K^u,W^s_\delta(x)]\subset B_m(x,2\delta)\cap f^{-m}(K_{J,\gamma}).$$

Replace $f,A$ by $f^{-1},A^{-1}$ in the above proof, we can obtain a subset $K^\prime_{J^\prime,\gamma} $ and $N_2\geq N$ such that for  any $x\in D(N,\theta)$,  $y\in B_m^{-}(x,2\delta)\cap f^{m}(K^\prime_{J^\prime,\gamma})$, we have
\[ \prod_{j=0}^{k-1}\|\mathcal{A}_{f^{-jN_2}y}^{-N_2}\|\cdot\|\left(\mathcal{A}_{f^{-jN_2}y}^{-N_2}\right)^{-1}\|\leq  e^{kN_2(\theta+3\gamma)}, \quad \forall k\geq 1,\]
where $B_m^{-}(x,2\delta)=\{y: d(f^ix,f^iy)<2\delta,~\forall -m+1\leq i\leq 0\}.$  We can also obtain a subset $K^s\subset W^s_\delta(x)$ such that 
\[\mu^s(K^s)>0, \quad\mbox{and} ~~ [W^u_\delta(x),K^s]\subset B_m^-(x,2\delta)\cap f^{m}(K^\prime_{J^\prime,\gamma}).\]

Let $N_*=N_1N_2,~\theta_*=\theta+3\gamma,$ and
\[E_m(x)=B_m(x,2\delta)\cap f^{-m}(K_{J,\gamma})\cap B_m^{-}(x,2\delta)\cap f^{m}(K^\prime_{J^\prime,\gamma}). \]
Then  we conclude that for any $x\in D(N,\theta), y\in E_m(x)$, one has $y\in D(N_*,\theta_*)$. 
Since $E_m(x)\supset[K^u,K^s]$, $\mu^u(K^u)>0,\mu^s(K^s)>0$ and $\mu$ has local product structure,  it follows that $\mu(E_m(x))>0$. Furthermore, by \cite[p.140]{Hirsch1970Stable}(see also \cite[Lemma 4.2]{Bowen470}), there exists $0<\beta<1$ such that $B_m(x,2\delta)\cap B_m^-(x,2\delta)\subset B(x,\beta^m)$. Thus it follows from $E_m(x)\subset D(N_*,\theta_*)\cap B(x,\beta^m)$ and $\mu(E_m(x))>0$ for every $m\geq1$ that $x\in \text{supp}(\mu|{D(N_*,\theta_*)}).$
 \end{proof}

  Now we can extend C to $D(N,\theta)$.
  \begin{proposition}\label{Prop 4.4}
       For any $\theta<\tau\alpha$ and $N\geq 1,$ there exists an $\alpha$-H\"older continuous map $\widehat{C}$ defined on $D(N,\theta)$ such that for any $n\geq 1$ and any  $x,f^n(x)\in D(N,\theta)$, we have
       \[\mathcal{A}_x^n=\widehat{C}(f^nx)\widehat{C}(x)^{-1}. \]
  \end{proposition}
 \begin{proof}
 	For any $\theta<\tau\alpha$ and $N\geq 1,$ Lemma \ref{Lemma 4.3} gives $N_*,\theta_*$ such that 
 	$$D(N,\theta)\subset\text{supp}(\mu|{D(N_*,\theta_*)}).$$
 	By Lemma \ref{Lemma 4.2}, there is an $\alpha$-H\"older continuous map $\widehat{C}$ defined on  $\text{supp}(\mu|{D(N_*,\theta_*)})$ which coincides $\mu$-a.e.  on   $\text{supp}(\mu|{D(N_*,\theta_*)})$ with C.  
 	
 	   Fix a $\varepsilon<\tau\rho-\theta_*.$ Now given any $x, f^nx\in D(N,\theta)$, take $N^\prime=N^\prime(n)$ large enough such that $R^{4n}<e^{N^\prime \varepsilon}$, where $R=\max_{y\in M}\{\|A(y)\|,\|A(y)^{-1}\|\}$. We may also assume that $N^\prime$ can be divided by $N_*$ so that 
 	   \[D(N_*,\theta_*)\subset D(N^\prime,\theta_*). \]
 	   For any $y\in M,$  $\|\mathcal{A}_{f^ny}^{N^\prime}\|=\|\mathcal{A}_{f^ny}^{N^\prime}\mathcal{A}_{y}^{n}(\mathcal{A}_y^n)^{-1}\|=\|\mathcal{A}_{f^{N^\prime}y}^n\mathcal{A}_y^{N^\prime}(\mathcal{A}_y^n)^{-1}\|\leq R^{2n}\|\mathcal{A}_y^{N^\prime}\|$. Similarly $\|(\mathcal{A}_{f^ny}^{N^\prime})^{-1}\|\leq R^{2n}\|(\mathcal{A}_{y}^{N^\prime})^{-1}\|$. It follows that  for any $y\in D(N^\prime,\theta_*)$ and $k\geq 1$,
 	  \begin{align*}
 	   \prod_{i=0}^{ k-1}\|\mathcal{A}_{f^{iN^\prime}(f^ny)}^{N^\prime}\|\|(\mathcal{A}_{f^{iN^\prime}(f^ny)}^{N^\prime})^{-1}\|&\leq R^{4n}\|\mathcal{A}_{f^{iN^\prime}(y)}^{N^\prime}\|\|(\mathcal{A}_{f^{iN^\prime}(y)}^{N^\prime})^{-1}\|\\
 	   &\leq R^{4nk}e^{kN^\prime\theta_*}\leq e^{kN^\prime(\theta_*+\varepsilon)}. 
 	  \end{align*}
 	   Replace $f,A$ by $f^{-1},A^{-1}$, the dual inequality can be proved analogously. Thus we conclude that $f^n(D(N^\prime,\theta_*)) \subset D(N^\prime,\theta_*+\varepsilon).$ Hence
 	   \begin{equation}\label{4.4}
 	   f^n(\text{supp}(\mu|{D(N^\prime,\theta_*)}))=\text{supp}(\mu|{f^n(D(N^\prime,\theta_*))}) \subset \text{supp}(\mu|{D(N^\prime,\theta_*+\varepsilon)}).
 	   \end{equation}
 	   By Lemma \ref{Lemma 4.2},, there exists a  continuous map ${C}^\prime$ defined on $\text{supp}(\mu|{D(N^\prime,\theta_*+\varepsilon)})$ which coincides $\mu$-a.e.  on   $\text{supp}(\mu|{D(N^\prime,\theta_*+\varepsilon)})$ with C.  Since $$\text{supp}(\mu|{D(N_*,\theta_*)})\subset \text{supp}(\mu|{D(N^\prime,\theta_*+\varepsilon)}),$$
 	   and $\widehat{C},C^\prime$  coincide $\mu$-a.e.  on   $\text{supp}(\mu|{D(N_*,\theta_*)})$ with C, one has 
 	   \begin{equation}\label{4.5}
 	   \widehat{C}(z)=C^\prime(z),\quad\forall z\in  D(N,\theta)\subset\text{supp}(\mu|{D(N_*,\theta_*)}).
 	   \end{equation}
 	   By \eqref{1.2} and \eqref{4.4},  for $\mu$-a.e. $y\in \text{supp}(\mu|{D(N^\prime,\theta_*)})$,  we have
 	  \begin{equation}\label{4.6}
 	   \mathcal{A}_y^n=C(f^ny)C(y)^{-1}=C^\prime(f^ny)C^\prime(y)^{-1}. 
 	  \end{equation}
 	   Take a sequence $\{y_k\}_{k\geq1}\subset \text{supp}(\mu|{D(N^\prime,\theta_*)})$ such that $y_k\to x$ as $k\to \infty.$
 	   Then by \eqref{4.6} and  \eqref{4.5}, we obtain 
 	   \[\mathcal{A}_x^n=C^\prime(f^nx)C^\prime(x)^{-1}=\widehat{C}(f^nx)\widehat{C}(x)^{-1}. \]
 	   This completes the proof. 	   
    \end{proof}
\subsection{Periodic obstructions of $\mathcal{A}$} 	   The key proposition we will prove in this subsection is that the measurable coboundary implies the periodic obstructions of $\mathcal{A}$.

\begin{proposition}\label{Proposition 4.5}
	Suppose that $\mathcal{A}$ is a measurable coboundary, that is, $\mathcal{A}$ satisfies the equation \eqref{1.2}. Then 
	\[\mathcal{A}_p^n=Id,\quad \forall p=f^n(p),\forall n\geq 1.\]
\end{proposition}

Before the proof of Proposition \ref{Proposition 4.5},  we first  estimate the norm of $\mathcal{A}$ along  an orbit segment close to a periodic one.
Let $p=f^J(p)$ be a periodic point for some $J\geq 1$,  denote $p_j=f^j(p),$ and let $\mu_p=\frac{1}{J}\sum_{i=0}^{J-1}\delta_{f^i(p)}$ be the corresponding periodic measure. Denote by $\lambda_+(p):=\lambda_+(\mathcal{A},\mu_p)$ and $\lambda_-(p):=\lambda_-(\mathcal{A},\mu_p)$ the upper and lower Lyapunov exponent of $\mathcal{A}$ with respect to $\mu_p$ respectively. 
For any $\varepsilon>0,$ define the Lyapunov norm $\|\cdot\|_{p_j}$ by
\[ \|u\|_{p_j}=\sum_{i=0}^{\infty}\|\mathcal{A}_{p_j}^iu\|e^{-(\lambda_+(p)+\varepsilon)i}+ \sum_{i=1}^{\infty}\|\mathcal{A}_{p_j}^{-i}u\|e^{(\lambda_-(p)-\varepsilon)i}. \] 
Since $p$ is a periodic point, we have that $\|\cdot\|_{p_j}$ is  uniformly equivalent to $\|\cdot\|$ for $p_j\in \mathcal{O}(p).$
Then similar to the proof of Lemma \ref{lemma 3.1} (or by Lemma 4.1 of \cite{Kalinin16}), we have 
\begin{lemma}\label{Lemma 4.6}
	Let $p$ be a periodic point of $f$. Then for any $0<\varepsilon<\frac{1}{2}\tau\alpha$, there exist  $\bar{\delta}=\bar{\delta}(p,\varepsilon)>0$ and $c=c(p,\varepsilon)>0$  such that for any  $x\in M$, $n\geq 1$ and $0<\delta<\bar{\delta}$ satisfying $d(f^jx,f^jp)\leq \delta e^{-\tau\min\{j,n-j\}},~j=0\cdots,n$, we have
	\[c^{-1}\cdot e^{j(\lambda_-(p)-2\varepsilon) }\leq m(\mathcal{A}_x^j)\leq\|\mathcal{A}_x^j\|\leq c\cdot e^{j(\lambda_+(p)+2\varepsilon) },\]
	\[c^{-1}\cdot e^{ (n-j)(\lambda_-(p)-2\varepsilon)}\leq m(\mathcal{A}_{x_j}^{n-j})\leq\|\mathcal{A}_{x_j}^{n-j}\|\leq c\cdot e^{(n-j)(\lambda_+(p)+2\varepsilon)},\]
	where $m(B):=\|B^{-1}\|^{-1}$.
\end{lemma}

  Moreover, if $\lambda_+(p)=\lambda_-(p)=0$, we can estimate the distortions.
  \begin{lemma}\label{Lemma 4.7}
  	Suppose that $\lambda_+(p)=\lambda_-(p)=0$. Then for any~$0<\varepsilon<\frac{1}{4}\tau\alpha$, there exist~$\widetilde{\delta}=\widetilde{\delta}(p,\varepsilon)>0,$ such that for any $0<\delta<\widetilde{\delta}$,  $x\in M$ and any $n\geq 1$ satisfying~$d(f^jx,f^jp)\leq \delta e^{-\tau\min\{j,n-j\}}$ for $j=0,\cdots,n$,
  		\[\frac{1}{2}\leq\frac{\|\mathcal{A}_p^n\|}{\|\mathcal{A}_x^n\|}\leq 2 \quad\mbox{and}\quad \frac{1}{2}\leq\frac{\|(\mathcal{A}_p^n)^{-1}\|}{\|(\mathcal{A}_x^n)^{-1}\|}\leq 2. \]
  \end{lemma}
  \begin{proof}
  	Denote $x_j=f^j(x)$ and $p_j=f^j(p).$ Then
  	\begin{equation}\label{4.7.0}
  	\begin{split}
  	\mathcal{A}_x^n-\mathcal{A}_p^n
  	= &  ~\mathcal{A}_{x_1}^{n-1}\circ(A(x_0)-A(p_0))+(\mathcal{A}_{x_1}^{n-1}-\mathcal{A}_{p_1}^{n-1})\circ A(p_0)\\
  	= & ~\mathcal{A}_{x_1}^{n-1}\circ(A(x_0)-A(p_0))+ \mathcal{A}_{x_2}^{n-2}\circ(A(x_1)-A(p_1))\circ A(p_0)+\\
  	& (\mathcal{A}_{x_2}^{n-2}-\mathcal{A}_{p_2}^{n-2})\circ\mathcal{A}_p^2\\
  	= & ~\cdots=\sum_{j=0}^{n-1}\mathcal{A}_{x_{j+1}}^{n-j-1}\circ(A(x_j)-A(p_j))\circ\mathcal{A}_p^j.
  	\end{split}
  	\end{equation}
  	Since   $\lambda_+(p)=\lambda_-(p)=0$, by Lemma \ref{Lemma 4.6} , for any $0\leq j\leq n$, 
  	$$\|\mathcal{A}_p^j\|\leq  c\cdot e^{2\varepsilon j }~\mbox{and}~ \|\mathcal{A}_p^j\|\leq\|(\mathcal{A}_{p_{j}}^{n-j})^{-1}\|\cdot\|\mathcal{A}_p^ n\|\leq c\cdot e^{2\varepsilon (n-j)}\|\mathcal{A}_p^ n\|.$$
  	Note that $\|A(x_j)-A(p_j)\|\leq c_0\delta^\alpha e^{-\tau\alpha \min\{j,n-j\}}.$ Denote $m=\left\lfloor \frac{n}{2}\right\rfloor,$ then Lemma \ref{Lemma 4.6} and  the fact $\varepsilon<\frac{1}{4}\tau\alpha$ give 
  	\begin{align*}
  	& \sum_{j=0}^{m}\|\mathcal{A}_{x_{j+1}}^{n-j-1}\|\cdot\|A(x_j)-A(p_j)\|\cdot\|\mathcal{A}_p^j\|\\
  	\leq & ~ \sum_{j=0}^{m}\|\mathcal{A}_x^n\|\cdot\|(\mathcal{A}_x^{j+1})^{-1}\| \cdot\|A(x_j)-A(p_j)\|\cdot\|\mathcal{A}_p^j\|\\
  	\leq &~ \|\mathcal{A}_x^n\|\cdot\sum_{j=0}^{m}ce^{2\varepsilon (j+1)}\cdot c_0\delta^\alpha e^{-
  		\tau\alpha j}\cdot ce^{2\varepsilon j}\\
  	\leq &~ \tilde{c}\delta^\alpha\cdot\|\mathcal{A}_x^n\|,
  	\end{align*}
  	and
  	\begin{align*}
  	& \sum_{j=m+1}^{n-1}\|\mathcal{A}_{x_{j+1}}^{n-j-1}\|\cdot\|A(x_j)-A(p_j)\|\cdot\|\mathcal{A}_p^j\|\\
  	\leq &~ \sum_{j=m+1}^{n-1}ce^{2\varepsilon (n-j-1)}\cdot c_0\delta^\alpha e^{-\lambda\alpha (n-j)}\cdot ce^{2\varepsilon (n-j)}\|\mathcal{A}_p^ n\|\\
  	\leq &~ \tilde{c}\delta^\alpha \|\mathcal{A}_p^ n\|,
  	\end{align*}
  	where $\tilde{c}=c_0c^2e^{2\varepsilon}/{(1-e^{4\varepsilon-\tau\alpha})}$. Therefore, 
  	\begin{equation*}
  	\begin{split}
   \|\mathcal{A}_x^n-\mathcal{A}_p^n\| &\leq \sum_{j=0}^{n-1}\|\mathcal{A}_{x_{j+1}}^{n-j-1}\|\cdot\|A(x_j)-A(p_j)\|\cdot\|\mathcal{A}_p^j\|\\
  	& \leq  \tilde{c}\delta^\alpha(\|\mathcal{A}_x^n\|+\|\mathcal{A}_p^ n\|).    
  	\end{split}   
  	\end{equation*}
  	Take $\widetilde{\delta}$ small enough such that $\tilde{c}\widetilde{\delta}^\alpha<\frac{1}{3}.$ Then for any $0<\delta<\widetilde{\delta},$ we obtain 
  	$$\frac{1}{2}\leq\frac{\|\mathcal{A}_p^n\|}{\|\mathcal{A}_x^n\|}\leq 2 .$$
  	Using
  	\begin{align*}
  	&(\mathcal{A}_x^n)^{-1}-(\mathcal{A}_p^n)^{-1}\\
  	= &  (A(x_0)^{-1}-A(p_0)^{-1})\circ(\mathcal{A}_{x_1}^{n-1})^{-1}+ A(p_0)^{-1}\circ ((\mathcal{A}_{x_1}^{n-1})^{-1}-(\mathcal{A}_{p_1}^{n-1})^{-1})\\
  	= & ~\cdots=\sum_{j=0}^{n-1}(\mathcal{A}_p^j)^{-1}\circ(A(x_j)^{-1}-A(p_j)^{-1})\circ(\mathcal{A}_{x_{j+1}}^{n-j-1})^{-1},
  	\end{align*}
  	the second  conclusion can be proved analogously.
  \end{proof}

  Lemma \ref{Lemma 4.7} holds under the condition $\lambda_+(p)=\lambda_-(p)=0$. In general, if $\lambda_\pm(p)\neq 0,$   the conclusion of Lemma \ref{Lemma 4.7} may not hold for all $n\in N$. However,  we shall show that if the distance of $f^i(x)$ and $f^i(p)$ are much closer, then  there exist infinitely many n such that the same conclusion holds.   We will use the following result by S. Gou\"{e}zel and A. Karlsson \cite{Gouezel&Karlsson}. 
\begin{proposition}[\cite{Gouezel&Karlsson}, Theorem 1.1 and Remark 1.2]\label{prop 4.8}
	Let $a_n(x)$ be an integrable subadditive cocycle with exponent $\lambda$ relative to an ergodic system $(M,f,\nu)$. Then for any $\rho>0,$ there exists a sequence $\varepsilon_i\to 0,$ a subset $E\subset M$ with $\nu(E)>1-\rho$, and a subset $S\subset \mathbb{N}$ with $\overline{Dens}(S)>1-\rho$ such that for any $x\in E$ and any $n\in S$,
	\[a_n(x)-a_{n-i}(f^{i}x)\geq (\lambda-\varepsilon_i)i,\quad \forall 0\leq i\leq n, \]
	where $\overline{Dens}(S):=\limsup_{N\to\infty}|S\cap[0,N-1]|/N.$ 
\end{proposition} 
  Consider the subadditive cocycles  $a_n(x)=\log\|\mathcal{A}_x^n\|$  and $\widetilde{a}_n(x)=\log\|(\mathcal{A}_x^n)^{-1}\|$.   Then the following  corollary can be deduced directly from Proposition \ref{prop 4.8}.
 \begin{corollary}\label{cor 4.9}
 	Let $p$ be a periodic point of $f$. Then for any $\rho>0$, there exists a sequence $\varepsilon_i\to 0$ and  a subset $S_p\subset \mathbb{N}$ with $\overline{Dens}(S_p)>1-\rho$ such that for any $n\in S_p$,
 	\[\|\mathcal{A}_{p_{i}}^{n-i}\|\leq\|\mathcal{A}_p^n\|e^{(-\lambda_+(p)+\varepsilon_i)i},\quad \forall 0\leq i\leq n, \]
 	\[\|(\mathcal{A}_{p_{i}}^{n-i})^{-1}\|\leq\|(\mathcal{A}_p^n)^{-1}\|e^{(\lambda_-(p)+\varepsilon_i)i},\quad \forall 0\leq i\leq n. \]
 \end{corollary}

 Now  we estimate the distortion along certain orbit segment.
\begin{lemma}\label{Lemma 4.10}
	Let $p$ be a periodic point of $f$.  Then for any~$0<\varepsilon<\frac{1}{6}\tau\alpha, \rho>0,$ there exist~$\widehat{\delta}=\widehat{\delta}(p,\varepsilon,\rho)>0$ and  a subset $S_p\subset \mathbb{N}$ with $\overline{Dens}(S_p)>1-\rho$ such that for any $n\in S_p$, $0<\delta<\widehat{\delta}$ and $x\in M$  satisfying~$d(f^jx,f^jp)\leq \delta e^{-\frac{1}{2}\tau j}$ for $j=0,\cdots,n$, we have
	\[\frac{1}{2}\leq\frac{\|\mathcal{A}_p^n\|}{\|\mathcal{A}_x^n\|}\leq 2 \quad\mbox{and}\quad \frac{1}{2}\leq\frac{\|(\mathcal{A}_p^n)^{-1}\|}{\|(\mathcal{A}_x^n)^{-1}\|}\leq 2. \]
\end{lemma}
\begin{proof}
	 We only prove the first conclusion, the second one can be proved in a similar fashion.
	 
	For any  $\rho>0,$ let $\varepsilon_i\to 0$ be given by  Corollary \ref{cor 4.9}. Then given any $0<\varepsilon<\frac{1}{6}\tau\alpha$, we may choose $L\geq 1$ large enough such that for any $i\geq L,$ one has $\varepsilon_i<\varepsilon$. Let the  subset $S_p\subset \mathbb{N}$ be given by Corollary \ref{cor 4.9}. Without loss of generality, we may assume $S_p\subset [L,\infty)$.  Then for any $n\in S_p,$
	\begin{equation}\label{4.7.1}
	\|\mathcal{A}_{p_{i}}^{n-i}\|\leq\|\mathcal{A}_p^n\|e^{(-\lambda_+(p)+\varepsilon)i},\quad \forall L\leq i\leq n.\
	\end{equation}
By \eqref{4.7.0}
	\begin{align*}
	\mathcal{A}_p^n-\mathcal{A}_x^n=\sum_{j=0}^{n-1}\mathcal{A}_{p_{j+1}}^{n-j-1}\circ(A(p_j)-A(x_j))\circ\mathcal{A}_x^j.
	\end{align*}
	
	  Note that $\|A(x_j)-A(p_j)\|\leq c_0\delta^\alpha e^{-\frac{1}{2}\tau\alpha j}.$  By   Lemma \ref{Lemma 4.6},
		\begin{align*}
	& \sum_{j=0}^{L-1}\|\mathcal{A}_{p_{j+1}}^{n-j-1}\|\cdot\|A(p_j)-A(x_j)\|\cdot\|\mathcal{A}_x^j\|\\
	\leq & ~ \sum_{j=0}^{L-1}\|\mathcal{A}_p^n\|\cdot\|(\mathcal{A}_{p}^{j+1})^{-1}\| \cdot\|A(p_j)-A(x_j)\|\cdot  \|\mathcal{A}_x^j\|\\
	\leq &~\sum_{j=0}^{L-1}\|\mathcal{A}_p^n\|\cdot \left(c\cdot e^{(j+1)(-\lambda_-(p)+2\varepsilon) }\right)\cdot \left(c_0\delta^\alpha e^{-\frac{1}{2}\tau\alpha j}\right)\cdot \left(c\cdot e^{j(\lambda_+(p)+2\varepsilon) }\right)\\
	\leq &~ \widetilde{c}_1\delta^\alpha\cdot\|\mathcal{A}_p^n\|,
	\end{align*}
	where $\widetilde{c}_1=c_0c^2L e^{(\lambda_+(p)-\lambda_-(p)+4\varepsilon-\frac{1}{2}\tau\alpha)L }$.  By   \eqref{4.7.1}, Lemma \ref{Lemma 4.6} and  the fact $\varepsilon<\frac{1}{6}\tau\alpha$, 
	\begin{align*}
	& \sum_{j=L}^{n-1}\|\mathcal{A}_{p_{j+1}}^{n-j-1}\|\cdot\|A(p_j)-A(x_j)\|\cdot\|\mathcal{A}_x^j\|\\
	\leq & ~ \sum_{j=L}^{n-1}\|\mathcal{A}_p^n\|e^{(-\lambda_+(p)+\varepsilon)(j+1)} \cdot\left(c_0\delta^\alpha e^{-\frac{1}{2}\tau\alpha j}\right)\cdot  \left(c\cdot e^{j(\lambda_+(p)+2\varepsilon) }\right)\\
	\leq &~ \|\mathcal{A}_p^n\|e^{-\lambda_+(p)+\varepsilon}c_0\delta^\alpha\cdot\sum_{j=L}^{n-1}e^{(3\varepsilon-\frac{1}{2}\tau\alpha) j}\\
	\leq &~ \widetilde{c}_2\delta^\alpha\cdot\|\mathcal{A}_p^n\|,
	\end{align*}
	where $\widetilde{c}_2= \frac{e^{-\lambda_+(p)+\varepsilon}}{1-e^{3\varepsilon-\tau\alpha/2}}\cdot c_0$.
   Therefore, 
	\begin{equation*}
	\begin{split}
	 \|\mathcal{A}_p^n-\mathcal{A}_x^n\| 
	\leq \sum_{j=0}^{n-1}\|\mathcal{A}_{p_{j+1}}^{n-j-1}\|\cdot\|A(p_j)-A(x_j)\|\cdot\|\mathcal{A}_x^j\|
	 \leq  (\widetilde{c}_1+\widetilde{c}_2)\delta^\alpha\|\mathcal{A}_p^n\|.    
	\end{split}   
	\end{equation*}
	Take $\widetilde{\delta}$ small enough such that $(\widetilde{c}_1+\widetilde{c}_2)\widetilde{\delta}^\alpha<\frac{1}{2}.$ Then for any $0<\delta<\widetilde{\delta},$ we conclude 
   \[\frac{1}{2}\|\mathcal{A}_x^n\|\leq\frac{2}{3}\|\mathcal{A}_x^n\|\leq\|\mathcal{A}_p^n\|\leq 2\|\mathcal{A}_x^n\|.\]
\end{proof}

Now we prove  Proposition \ref{Proposition 4.5}.
\begin{proof}[Proof of Proposition \ref{Proposition 4.5}]
	We begin by finding a periodic point $p_1=f^k(p_1)$ such that  $\mathcal{A}_{p_1}^k=Id.$ By Kalinin and Sadovskaya's result \cite[Theorem 1.4]{Kalinin16}, for any $0<\theta<\tau\alpha, $ there exists a periodic point $p_1=f^k(p_1)$ such that
	\[\big|\lambda_+(\mathcal{A},\mu)-\frac{1}{k}\log\|\mathcal{A}_{p_1}^k\|\big|\leq \theta,~\mbox{and}~~~~\big|\lambda_-(\mathcal{A},\mu)-\frac{1}{k}\log\|(\mathcal{A}_{p_1}^k)^{-1}\|^{-1}\big|\leq \theta. \]
 	Since $\lambda_+(\mathcal{A},\mu)=\lambda_-(\mathcal{A},\mu)=0$, we have
 	\[\|\mathcal{A}_{p_1}^k\|\leq e^{k\theta},~\mbox{and}\quad \|(\mathcal{A}_{p_1}^k)^{-1}\|\leq e^{k\theta}. \]
 	It follows that $p_1\in D(k,\theta)$. Then  by Proposition \ref{Prop 4.4}, 
 	\[\mathcal{A}_{p_1}^k=\widehat{C}(f^kp_1)\widehat{C}(p_1)^{-1}=\widehat{C}(p_1)\widehat{C}(p_1)^{-1}=Id. \]
 	
 	Now for any periodic point $p_2=f^m(p_2)$,    in order  to prove  $\mathcal{A}_{p_2}^m=Id$, it's enough to show $\lambda_+(p_2)=\lambda_-(p_2)=0$. Indeed, if $\lambda_+(p_2)=\lambda_-(p_2)=0$, then for any $0<\theta<\tau\alpha$, there exists $n\in\mathbb{N}$ large enough such that $	\|\mathcal{A}_{p_2}^{nm}\|\leq e^{nm\theta}$ and $ \|(\mathcal{A}_{p_2}^{nm})^{-1}\|\leq e^{nm\theta },$ which implies $p_2\in D(nm,\theta)$. Since  $p_2=f^m(p_2)$, by  Proposition \ref{Prop 4.4},
 \[\mathcal{A}_{p_2}^{m}=\widehat{C}(f^{m}p_2)\widehat{C}(p_2)^{-1}=Id. \]	
 
 Assume $\lambda_+(p_2)>0$. To get a contradiction, we will use the fact that the topological mixing Anosov diffeomorphism $f$ satisfies the {\em specification property} \cite{Bowen1974equilibrium,katok1995}: For any $\delta>0$, there exists $N=N(\delta)\geq  1$ such that for any points $x_1,x_2,\cdots,x_n$ 
 	and any intervals of integers $I_1,I_2,\cdots,I_n\subset [a,b]$ with $d(I_i,I_j)\geq N$ for $i\neq j$, then there exists a periodic point $x=f^{b-a+2N}(x)$ such that $d(f^j(x),f^j(x_i))< \delta$ for $j\in I_i.$   Moreover,  by the following lemma, the distance of $f^j(x)$ and $f^j(x_i)$ can be exponentially close.
 	\begin{lemma}[Proposition 6.4.16 of \cite{katok1995}]\label{Lemma 4.11}
 		There  exists $\delta^\prime>0 $ and $c^\prime\geq 1$ such that for any $0<\delta<\delta^\prime$ and any $x,y\in M$ with $d(f^ix,f^iy)<\delta$ for $i=0,\cdots,n$, then in fact 
 		\[ d(f^ix,f^iy)<c^\prime\delta e^{-\tau\min\{i,n-i \}}.\]
 	\end{lemma}
 	Now given  any  $0<\varepsilon<\frac{1}{6}\min\{\tau\alpha,\lambda_+(p_2)\},$ take $0<\rho<\frac{1}{3m}$, and let $\delta>0$ be given such that $$c^\prime\delta<\min\{\bar{\delta}(p_2,\varepsilon),\widetilde{\delta}(p_1,\varepsilon),\widehat{\delta}(p_2,\varepsilon,\rho),\delta^\prime\} ,$$ where $\bar{\delta}(p_2,\varepsilon), \widetilde{\delta}(p_1,\varepsilon),\widehat{\delta}(p_2,\varepsilon,\rho),\delta^\prime$ are given by Lemma \ref{Lemma 4.6}, Lemma \ref{Lemma 4.7},Lemma \ref{Lemma 4.10} and Lemma \ref{Lemma 4.11} respectively.
  Let $S_{p_2}\in\mathbb{N}$ be given by Lemma \ref{Lemma 4.10}. Since $\overline{Dens}(S_{p_2})>1-\rho>1-\frac{1}{3m},$ there are infinitely many $b\in \mathbb{N}$ such that $3mb\in S_{p_2}.$ 	Let $c_1=\max\limits_{x\in M}\{ \|A(x)\|,\|A(x)^{-1}\|\}$,  $c_2=c(p_2,\varepsilon)$ be given by Lemma \ref{Lemma 4.6}. Note that $\lambda_+(p_2)=\lim\limits_{b\to\infty}\frac{1}{bm}\log\|\mathcal{A}_{p_2}^{bm}\|$. We may  choose $b\in\mathbb{N}$ large enough such that $2bm\in S_{p_2}$ and
 	\begin{equation}\label{4.7}
 	\|\mathcal{A}_{p_2}^{2bm}\|\geq e^{2bm(\lambda_+(p_2)-\varepsilon) }	   >4{c_1}^{2N}c_2\cdot e^{bm(\lambda_+(p_2)+2\varepsilon)},
 	 \end{equation}
  	where $N=N(\delta)$ is given by the specification property.
   Then choose $a\in \mathbb{N}$ large enough such that
 \begin{equation}\label{4.9}
 2{c_1}^{2N}c_2\cdot e^{3bm(\lambda_+(p_2)+2\varepsilon)}\leq e^{2\varepsilon(ak+2N+3bm)},~~\mbox{and}
 \end{equation}	
 	 \begin{equation}\label{4.10}
 	2{c_1}^{2N}c_2\cdot e^{3bm(-\lambda_-(p_2)+2\varepsilon)}\leq e^{2\varepsilon(ak+2N+3bm)}.
 	\end{equation}
 	Then by the specification property, for $p_1$ and $f^{-ak-N}(p_2)$, there exists a periodic point $q=f^{ak+2N+3bm}(q)$ such that 
 	\[d(f^i(q),f^i(p_1))<\delta,~\forall 0\leq i\leq ak, ~\mbox{and}\]
 	\[d(f^{j+ak+N}(q),f^j(p_2))<\delta,~\forall 0\leq j\leq 3bm. \]
 	Denote $x=f^{ak+N}(q)$. Then  by  Lemma \ref{Lemma 4.11}, 
 	\[d(f^i(q),f^i(p_1))<c^\prime\delta e^{-\tau\min\{i,ak-i\}}<\widetilde{\delta}(p_1,\varepsilon)e^{-\tau\min\{i,ak-i\}},~\forall 0\leq i\leq ak, ~\mbox{and}\]
 	\[d(f^{j}(x),f^j(p_2))<c^\prime\delta e^{-\tau\min\{j,3bm-j\}}< \bar{\delta}(p_2,\varepsilon)e^{-\tau\min\{j,3bm-j\}},~\forall 0\leq j\leq 3bm. \]
 	Since $\mathcal{A}_{p_1}^{ak}=Id$, by Lemma \ref{Lemma 4.7},
 	\begin{equation*}
 	\|\mathcal{A}_{q}^{ak}\|\leq 2,~\mbox{and}~~\|(\mathcal{A}_{q}^{ak})^{-1}\|\leq 2.
 	\end{equation*}
 	By Lemma \ref{Lemma 4.6},  
 	$$\|\mathcal{A}_x^{3bm}\|\leq c_2\cdot e^{3bm(\lambda_+(p_2)+2\varepsilon)},\mbox{and~} \|(\mathcal{A}_x^{3bm})^{-1}\|\leq c_2\cdot e^{3bm(-\lambda_-(p_2)+2\varepsilon)}  .$$
 	Therefore,  using \eqref{4.9},
 \begin{align*}
 	\|\mathcal{A}_q^{ak+2N+3bm}\|
 	&\leq\|\mathcal{A}_{f^{3bm}x}^{N}\|\cdot\|\mathcal{A}_{x}^{3bm}\|\cdot\|\mathcal{A}_{f^{ak}q}^{N}\|\cdot\|\mathcal{A}_{q}^{ak}\|\\
 	 &\leq   2{c_1}^{2N}c_2\cdot e^{3bm(\lambda_+(p_2)+2\varepsilon)}\\
 	 &\leq e^{2\varepsilon(ak+2N+3bm)}.
 \end{align*}
 	Similarly, we can also get $	\|(\mathcal{A}_q^{ak+2N+3bm})^{-1}\|\leq e^{2\varepsilon(ak+2N+3bm)}.$ It follows that 
 	$$q\in D(ak+2N+3bm,2\varepsilon).$$
 	Then by Proposition \ref{Prop 4.4}, $\mathcal{A}_q^{ak+2N+3bm}=\widehat{C}(f^{ak+2N+3bm}(q))\widehat{C}(q)^{-1}=Id,$
  which implies
 	\begin{align*}
 	\|\mathcal{A}_{x}^{3bm}\|
 	\leq\|(\mathcal{A}_{q}^{ak})^{-1}\|\cdot\|(\mathcal{A}_{f^{ak}q}^{N})^{-1}\|\cdot\|(\mathcal{A}_{f^{3bm}x}^{N})^{-1}\|\leq 2{c_1}^{2N}.
 	\end{align*}
 	Hence by Lemma \ref{Lemma 4.6},
 	\begin{align*}
 	\|\mathcal{A}_{x}^{2bm}\|
 	\leq\|(\mathcal{A}_{x}^{3bm})\|\cdot\|(\mathcal{A}_{f^{2bm}x}^{bm})^{-1}\|\leq 2{c_1}^{2N}c_2\cdot e^{bm(\lambda_+(p_2)+2\varepsilon)}.
 	\end{align*}
 	Now for any $0\leq j\leq 2bm$, one has $3bm\geq \frac{j}{2}.$ Therefore, for any  $0\leq j\leq 2bm$,
 	\begin{align*}
 	d(f^{j}x,f^jp_2)<c^\prime\delta e^{-\tau\min\{j,3bm-j\}}\leq c^\prime\delta e^{-\tau\min\{j,\frac{j}{2}\}}\leq c^\prime\delta e^{-\frac{1}{2}\tau j}
 	\leq \widehat{\delta}(p_2,\varepsilon,\rho)e^{-\frac{1}{2}\tau j}. 
 	\end{align*}
 	By  the choice of b and Lemma \ref{Lemma 4.10},  one has 
 	\[\|\mathcal{A}_{p_2}^{2bm}\|\leq 2\|\mathcal{A}_{x}^{2bm}\|\leq  4{c_1}^{2N}c_2\cdot e^{bm(\lambda_+(p_2)+2\varepsilon)}.\]
  This contradicts   \eqref{4.7}. Hence  $\lambda_+(p_2)\leq 0.$   It can be proved in  a similar fashion that   $\lambda_-(p_2)\geq 0.$  Then we conclude that $\lambda_+(p_2)=\lambda_-(p_2)=0.$ This completes the proof of Proposition \ref{Proposition 4.5}.
\end{proof}

Now we finish the proof of Theorem  \ref{thm B}. By Proposition \ref{Proposition 4.5} and Theorem 1.4 of \cite{Kalinin16},	for any   ergodic $f$-invariant probability measure $\nu$, $\lambda_+(\mathcal{A},\nu)=\lambda_-(\mathcal{A},\nu)=0.$ 
Since
 \begin{align*}
\lambda_+(\mathcal{A},\nu)-\lambda_-(\mathcal{A},\nu)
&=\lim\limits_{n\to \infty}\frac{1}{n}\int\log\left(\|\mathcal{A}_{x}^{n} \|\|(\mathcal{A}_x^n)^{-1}\|\right)d\nu\\
&=\lim\limits_{n\to \infty}\frac{1}{n}\int\log\left(\|({\mathcal{A}_{x}^{-n}})^{-1} \|\|\mathcal{A}_{x}^{-n} \|\right)d\nu,
\end{align*}
 it follows from \cite{Schreiber98} that  
 \[\lim\limits_{n\to\infty} \frac{1}{n}\max_{x\in M}\log\left(\|\mathcal{A}_{x}^{n} \|\|(\mathcal{A}_x^n)^{-1}\| \right) 
 =\sup \left\{\lambda_+(\mathcal{A},\nu)-\lambda_-(\mathcal{A},\nu): v\in \mathcal{E}(f)\right\} =0,\]
  \[  \lim\limits_{n\to\infty} \frac{1}{n}\max_{x\in M}\log\left(\|({\mathcal{A}_{x}^{-n}})^{-1} \|\|\mathcal{A}_{x}^{-n} \|\right)
 	   =\sup \left\{\lambda_+(\mathcal{A},\nu)-\lambda_-(\mathcal{A},\nu): v\in \mathcal{E}(f)\right\} 
 	   =0,\]
 where $\mathcal{E}(f)$ denotes the space of ergodic $f$-invariant Borel probability measures. Then for any $\varepsilon<\tau\alpha$, there exists $N\geq 1$, such that
 \[\|\mathcal{A}_{x}^{N} \|\|(\mathcal{A}_x^N)^{-1}\|\leq e^{\varepsilon N} ~\mbox{and}~\|({\mathcal{A}_{x}^{-N}})^{-1} \|\|\mathcal{A}_{x}^{-N} \|\leq e^{\varepsilon N},~\forall x\in M.\]
We conclude that $D(N,\varepsilon)=M.$ Hence by Proposition \ref{Prop 4.4}, we obtain the desired  $\alpha$-H\"older continuous map $\widehat{C}.$

\bibliographystyle{amsplain}
%


\end{document}